\documentclass[11pt]{article}
\usepackage[margin=1in]{geometry}
\usepackage{amsfonts,amsmath,amssymb}

\PassOptionsToPackage{obeyspaces}{url}
\usepackage[colorlinks=true,citecolor=blue,urlcolor=blue,linkcolor=blue,bookmarksopen=true]{hyperref}
\usepackage{amsthm}
\usepackage{bm}
\usepackage{tikz}

\usepackage{etoolbox}
\patchcmd{\thebibliography}{\leftmargin\labelwidth}{\leftmargin\labelwidth\addtolength\itemsep{-0.1\baselineskip}}{}{}

\usepackage[nameinlink,sort]{cleveref}

\newtheorem{theorem}{Theorem}
\newtheorem{lemma}[theorem]{Lemma}
\newtheorem{corollary}[theorem]{Corollary}

\crefname{conj}{conjecture}{conjectures}

\crefname{claim}{claim}{claims}

\newtheorem{prop}[theorem]{Proposition}
\crefname{prop}{proposition}{propositions}

\theoremstyle{definition}

\newtheorem{defn}[theorem]{Definition}
\crefname{defn}{definition}{definitions}

\crefname{remark}{remark}{remarks}

\crefname{question}{question}{questions}

\crefname{enumi}{part}{parts}

\crefname{equation}{eq.\!}{eqs.\!}

\numberwithin{theorem}{section}

\newcommand*{\eqdef}{\stackrel{\mbox{\normalfont\tiny def}}{=}}   
\newcommand*{\abs}[1]{\lvert #1\rvert}                
\newcommand*{\R}{\mathbb R}
\newcommand*{\Z}{\mathbb Z}
\DeclareMathOperator*{\E}{\mathbb E}

\DeclareMathOperator{\iso}{iso}
\DeclareMathOperator{\Iso}{Iso}

\newcommand*{\slicex}[2]{{#1\{\cdot,#2\}}}
\newcommand*{\slicey}[2]{{#1\{#2,\cdot\}}}
\newcommand*{\rootedIntervals}{\mathcal R}
\newcommand*{\lebesgue}{\mathcal L}

\newcommand*{\cube}[1]{\operatorname{\mathbf C}({#1})}
\newcommand*{\anticube}[1]{\operatorname{\mathbf A}({#1})}
\newcommand*{\ncube}[1]{\operatorname{\mathbf c}[{#1}]}
\newcommand*{\nanticube}[1]{\operatorname{\mathbf a}[{#1}]}
\newcommand*{\bcube}[1]{\operatorname{\mathbf{\hat c}}[{#1}]}
\newcommand*{\banticube}[1]{\operatorname{\mathbf{\hat a}}[{#1}]}
\newcommand*{\qiso}[1]{\operatorname{\mathbf q}[{#1}]}

\title{Isoperimetric inequalities in the $\ell_\infty$ cube and torus}

\author{%
    Ben Baker\thanks{Dept.\ of Mathematics \& Statistics, Auburn University, Auburn, AL, USA. \texttt{\{bzb0122,jgb0059\}@auburn.edu}.}
    \and
    Joseph Briggs\footnotemark[1]
    \and
    Manuel Fernandez V\thanks{Dept.\ of Mathematics, University of Southern California, Los Angeles, CA, USA. \texttt{manuelf7@usc.edu}.}
    \and
    Chris Wells\thanks{Dept.\ of Mathematics \& Statistics, Haverford College, Haverford, PA, USA. \texttt{cwells@haverford.edu}.}
}

\date{}

\begin{document}
\maketitle

\begin{abstract}
    We prove an exact isoperimetric inequality for cubes and tori under the $\ell_\infty$ metric.
    As a corollary, we recover the celebrated grid edge-isoperimetric inequality of Bollob\'as--Leader.
\end{abstract}

\section{Introduction}

For a set $A\subseteq\R^n$ and a positive number $\delta$, define its $\ell_\infty$-neighborhood of radius $\delta$ to be
\[
    N_\delta[A]\eqdef\{x\in\R^n:\lVert x-a\rVert_\infty<\delta\text{ for some $a\in A$}\}.
\]
Our main result is a complete determination of the isoperimetric function when restricted to the $n$-dimensional cube $[0,1]^n$.
Let $\lebesgue^n$ denote the $n$-dimensional Lebesgue measure.

\begin{theorem}\label[theorem]{mainCube}
    For any $0<\alpha\leq 1$ and $0<\delta\leq 1$,
    \[
        \min\bigl\{\lebesgue^n\bigl(N_\delta[A]\cap[0,1]^n\bigr):A\subseteq[0,1]^n,\ \lebesgue^n(A)\geq\alpha\bigr\}=\min_{t\in[n]}\bigl\{1,\bigl(\alpha^{1/t}+\delta\bigr)^t,1-\bigl((1-\alpha)^{1/t}-\delta\bigr)^t\bigr\}.
    \]
\end{theorem}
A class of optimizers here have the form either $A=C$ or $A=[0,1]^n\setminus C$ where $C=[0,\alpha^{1/t}]^t\times[0,1]^{n-t}$ for some $t$.
A depiction of \Cref{mainCube} is shown in \Cref{fig:phases}.

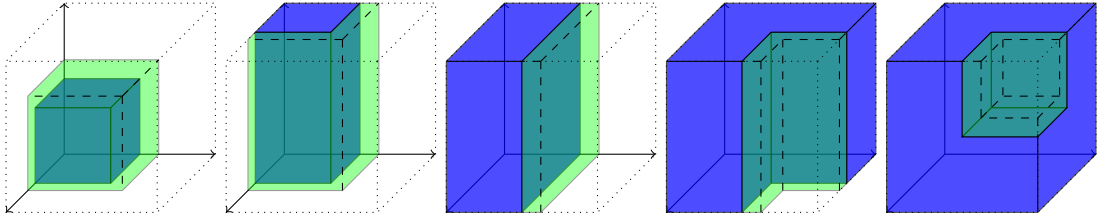
\begin{figure}[ht]
    \begin{center}
        \begin{tikzpicture}
            \draw[->] (0,0,0)--(2,0,0);
            \draw[->] (0,0,0)--(0,2,0);
            \draw[->] (0,0,0)--(0,0,2);

            \draw[fill=blue,opacity=0.7] (1,0,0)--(1,1,0)--(0,1,0)--(0,1,1)--(0,0,1)--(1,0,1)--cycle;
            \draw (1,1,1)--(1,1,0);
            \draw (1,1,1)--(1,0,1);
            \draw (1,1,1)--(0,1,1);

            \draw[fill=green,opacity=0.4] (1.25,0,0)--(1.25,1.25,0)--(0,1.25,0)--(0,1.25,1.25)--(0,0,1.25)--(1.25,0,1.25)--cycle;
            \draw[dashed] (1.25,1.25,1.25)--(1.25,1.25,0);
            \draw[dashed] (1.25,1.25,1.25)--(1.25,0,1.25);
            \draw[dashed] (1.25,1.25,1.25)--(0,1.25,1.25);

            \draw[dotted] (2,0,0)--(2,2,0)--(0,2,0)--(0,2,2)--(0,0,2)--(2,0,2)--cycle;
            \draw[dotted] (2,2,2)--(2,2,0);
            \draw[dotted] (2,2,2)--(2,0,2);
            \draw[dotted] (2,2,2)--(0,2,2);
        \end{tikzpicture}
        \begin{tikzpicture}
            \draw[->] (0,0,0)--(2,0,0);
            \draw[->] (0,0,0)--(0,2,0);
            \draw[->] (0,0,0)--(0,0,2);

            \draw[fill=blue,opacity=0.7] (1,0,0)--(1,2,0)--(0,2,0)--(0,2,1)--(0,0,1)--(1,0,1)--cycle;
            \draw (1,2,1)--(1,2,0);
            \draw (1,2,1)--(1,0,1);
            \draw (1,2,1)--(0,2,1);

            \draw[fill=green,opacity=0.4](1.25,0,0)--(1.25,2,0)--(1,2,0)--(1,2,1)--(0,2,1)--(0,2,1.25)--(0,0,1.25)--(1.25,0,1.25)--cycle;
            \draw[dashed] (1.25,2,1.25)--(1.25,2,0);
            \draw[dashed] (1.25,2,1.25)--(1.25,0,1.25);
            \draw[dashed] (1.25,2,1.25)--(0,2,1.25);

            \draw[dotted] (2,0,0)--(2,2,0)--(0,2,0)--(0,2,2)--(0,0,2)--(2,0,2)--cycle;
            \draw[dotted] (2,2,2)--(2,2,0);
            \draw[dotted] (2,2,2)--(2,0,2);
            \draw[dotted] (2,2,2)--(0,2,2);
        \end{tikzpicture}
        \begin{tikzpicture}
            \draw[->] (0,0,0)--(2,0,0);
            \draw[->] (0,0,0)--(0,2,0);
            \draw[->] (0,0,0)--(0,0,2);

            \draw[fill=blue,opacity=0.7] (1,0,0)--(1,2,0)--(0,2,0)--(0,2,2)--(0,0,2)--(1,0,2)--cycle;
            \draw (1,2,2)--(1,2,0);
            \draw (1,2,2)--(1,0,2);
            \draw (1,2,2)--(0,2,2);

            \draw[fill=green,opacity=0.4] (1.25,0,0)--(1.25,2,0)--(1,2,0)--(1,2,2)--(1,0,2)--(1.25,0,2)--cycle;
            \draw[dashed] (1.25,2,2)--(1.25,2,0);
            \draw[dashed] (1.25,2,2)--(1.25,0,2);
            \draw[dashed] (1.25,2,2)--(1,2,2);

            \draw[dotted] (2,0,0)--(2,2,0)--(0,2,0)--(0,2,2)--(0,0,2)--(2,0,2)--cycle;
            \draw[dotted] (2,2,2)--(2,2,0);
            \draw[dotted] (2,2,2)--(2,0,2);
            \draw[dotted] (2,2,2)--(0,2,2);
        \end{tikzpicture}
        \begin{tikzpicture}
            \draw[->] (0,0,0)--(2,0,0);
            \draw[->] (0,0,0)--(0,2,0);
            \draw[->] (0,0,0)--(0,0,2);

            \draw[fill=blue,opacity=0.7] (2,0,0)--(2,2,0)--(0,2,0)--(0,2,2)--(0,0,2)--(1,0,2)--(1,0,1)--(2,0,1)--cycle;
            \draw (1,2,2)--(1,2,1);
            \draw (1,2,2)--(1,0,2);
            \draw (1,2,2)--(0,2,2);
            \draw (2,2,1)--(1,2,1);
            \draw (2,2,1)--(2,2,0);
            \draw (2,2,1)--(2,0,1);
            \draw (1,2,1)--(1,0,1);

            \draw[fill=green,opacity=0.4] (2,0,1)--(2,2,1)--(1,2,1)--(1,2,2)--(1,0,2)--(1.25,0,2)--(1.25,0,1.25)--(2,0,1.25)--cycle;
            \draw[dashed] (1.25,2,2)--(1.25,2,1.25);
            \draw[dashed] (1.25,2,2)--(1.25,0,2);
            \draw[dashed] (1.25,2,2)--(1,2,2);
            \draw[dashed] (2,2,1.25)--(1.25,2,1.25);
            \draw[dashed] (2,2,1.25)--(2,2,1);
            \draw[dashed] (2,2,1.25)--(2,0,1.25);
            \draw[dashed] (1.25,2,1.25)--(1.25,0,1.25);

            \draw[dotted] (2,0,0)--(2,2,0)--(0,2,0)--(0,2,2)--(0,0,2)--(2,0,2)--cycle;
            \draw[dotted] (2,2,2)--(2,2,0);
            \draw[dotted] (2,2,2)--(2,0,2);
            \draw[dotted] (2,2,2)--(0,2,2);
        \end{tikzpicture}
        \begin{tikzpicture}
            \draw[->] (0,0,0)--(2,0,0);
            \draw[->] (0,0,0)--(0,2,0);
            \draw[->] (0,0,0)--(0,0,2);

            \draw[fill=blue,opacity=0.7] (2,0,0)--(2,2,0)--(0,2,0)--(0,2,2)--(0,0,2)--(2,0,2)--cycle;
            \draw (1,2,2)--(1,2,1);
            \draw (1,2,2)--(1,1,2);
            \draw (1,2,2)--(0,2,2);
            \draw (2,2,1)--(1,2,1);
            \draw (2,2,1)--(2,2,0);
            \draw (2,2,1)--(2,1,1);
            \draw (1,2,1)--(1,1,1);
            \draw (2,1,1)--(1,1,1)--(1,1,2)--(2,1,2)--cycle;
            \draw (2,1,2)--(2,0,2);

            \draw[fill=green,opacity=0.4] (2,1,1)--(2,2,1)--(1,2,1)--(1,2,2)--(1,1,2)--(2,1,2)--cycle;
            \draw[dashed] (1.25,2,2)--(1.25,2,1.25);
            \draw[dashed] (1.25,2,2)--(1.25,1.25,2);
            \draw[dashed] (1.25,2,2)--(1,2,2);
            \draw[dashed] (2,2,1.25)--(1.25,2,1.25);
            \draw[dashed] (2,2,1.25)--(2,2,1);
            \draw[dashed] (2,2,1.25)--(2,1.25,1.25);
            \draw[dashed] (1.25,2,1.25)--(1.25,1.25,1.25);
            \draw[dashed] (2,1.25,1.25)--(1.25,1.25,1.25)--(1.25,1.25,2)--(2,1.25,2)--cycle;
            \draw[dashed] (2,1.25,2)--(2,1,2);

            \draw[dotted] (2,0,0)--(2,2,0)--(0,2,0)--(0,2,2)--(0,0,2)--(2,0,2)--cycle;
            \draw[dotted] (2,2,2)--(2,2,0);
            \draw[dotted] (2,2,2)--(2,0,2);
            \draw[dotted] (2,2,2)--(0,2,2);
        \end{tikzpicture}
    \end{center}
    \caption{\label{fig:phases}A depiction of the phases of the isoperimetric inequality in \Cref{mainCube}.}
\end{figure}

A similar result holds for the $n$-dimensional torus $\R^n/\Z^n$:
\begin{theorem}\label[theorem]{mainTorus}
    For any $0<\alpha\leq 1$ and any $0<\delta\leq 1/2$,
    \[
        \min\bigl\{\lebesgue^n\bigl(N_\delta[A]/\Z^n\bigr):A\subseteq\R^n/\Z^n,\ \lebesgue^n(A)\geq\alpha\bigr\}=\min_{t\in[n]}\bigl\{1,\bigl(\alpha^{1/t}+2\delta\bigr)^t,1-\bigl((1-\alpha)^{1/t}-2\delta\bigr)^t\bigr\}.
    \]
\end{theorem}
A class of optimizers here have the form either $A=C+\Z^n$ or $A=\bigl([0,1)^n\setminus C\bigr)+\Z^n$ where $C=[0,\alpha^{1/t}]^t\times[0,1]^{n-t}$ for some $t$.

We actually prove a broad generalization of these two theorems which considers products of intervals and circles of various scales.
\medskip

These results yield approximate graph isoperimetric inequalities for strong products of paths and cycles.
For graphs $G_1,\dots,G_n$, the strong product $G_1\boxtimes\dots\boxtimes G_n$ has vertex-set $V(G_1)\times\dots\times V(G_n)$ and $(v_1,\dots,v_n)$ is adjacent to $(u_1,\dots,u_n)$ whenever, for each $i\in[n]$, either $v_i=u_i$ or $v_iu_i$ is an edge in $G_i$.
We write $G^{\boxtimes n}$ to denote the $n$-fold strong product of $G$ with itself.
For a graph $G$ and a subset $A\subseteq V(G)$, let $N_G[A]$ denote the closed neighborhood of $A$, that is all vertices in $A$ or adjacent to some vertex in $A$.
Let $C_k$ denote the cycle graph on $k$ vertices.
\begin{corollary}\label[corollary]{mainCycles}
    For any positive integers $k,n$, if $A\subseteq V(C_k^{\boxtimes n})$ has size $\abs A=\alpha k^n$, then
    \[
        \bigl\lvert N_{C_k^{\boxtimes n}}[A]\bigr\rvert\geq k^n\cdot\min_{t\in[n]}\bigl\{1,\bigl(\alpha^{1/t}+2/k\bigr)^t,1-\bigl((1-\alpha)^{1/t}-2/k\bigr)^t\bigr\}.
    \]
\end{corollary}

Let $P_k$ denote the path graph on $k$ vertices.
\begin{corollary}\label[corollary]{mainPaths}
    For any positive integers $k,n$, if $A\subseteq V(P_k^{\boxtimes n})$ has size $\abs A=\alpha k^n$, then
    \[
        \bigl\lvert N_{P_k^{\boxtimes n}}[A]\bigr\rvert\geq k^n\cdot\min_{t\in[n]}\bigl\{1,\bigl(\alpha^{1/t}+1/k\bigr)^t,1-\bigl((1-\alpha)^{1/t}-1/k\bigr)^t\bigr\}.
    \]
\end{corollary}
Again, we actually prove a broad generalization of these corollaries which applies to strong products of cycles and paths of various lengths.

We note that recently, in {\cite[Theorem 1.4]{wang2025_2D}}, Wang determined the minimum of $|N_{P_n \boxtimes P_m}[A]|$ for \emph{any} choice of $|A|$.
He also commented that the $n$-dimensional case $\boxtimes_{i=1}^n P_{m_i}$ ``appears to be very challenging''.
While \Cref{mainCube} addresses arbitrary dimensions in the sense of Problems 4.2 and 4.3 of \cite{wang2025_2D}, and also accounts for different path lengths, it does not fully generalize Theorem 1.4 of \cite{wang2025_2D}, which would require finer approximations.
In a similar vein, for the bi-infinite path $P_\infty$, Radcliffe--Veomett \cite{radcliffe_vertex} determined precisely the minimum of $|N_{P_\infty^{\boxtimes n}}[A]|$, for every $|A|$ and every $n$.
There is a tiny amount of overlap with their theorem and \Cref{mainCycles} in the special case where $|A|=a^n$ because the local structures of $C_k^{\boxtimes n}$ (for sufficiently large $k$) and $P_\infty^{\boxtimes n}$ agree everywhere, and the optimal sets are translations of subgrids $[a]^n$ in both cases.
However, this special case of both of our theorems follows more directly from the Brunn--Minkowski inequality.

\medskip

Lastly, we show that the celebrated isoperimetric inequality of Bollob\'as--Leader~\cite{bollobas_edge} (independently rediscovered for $\mathbb{R}^3$ by Chambers--Mouill\'e \cite{chambers2023_3D}) is a corollary of our results.
A set is called \emph{rectilinear} if it is a finite union of boxes.

\begin{theorem}[Bollob\'as--Leader~\cite{bollobas_edge}]
    For any rectilinear set $A\subseteq[0,1]^n$, there is a set $X$ of the form either $[0,\alpha]^t\times[0,1]^{n-t}$ or $[0,1]^n\setminus\bigl([0,\alpha]^t\times[0,1]^{n-t}\bigr)$ such that $\lebesgue^n(X)=\lebesgue^n(A)$ and the exposed surface area of $\partial X\setminus\partial[0,1]^n$ being at most the exposed surface area of $\partial A\setminus\partial[0,1]^n$.
\end{theorem}
We will make precise exactly what we mean by ``exposed surface area'' in \Cref{corollaries}. Note also that the application of a continuous isoperimetric result to graphs is not new. In \cite{bollobas_edge} the above theorem was used to deduce an edge-isoperimetric inequality in the finite grid graph, and our deductions of \Cref{mainCycles} and \Cref{mainPaths} from \Cref{mainTorus} and \Cref{mainCube} are essentially identical---indeed, our proofs of \Cref{mainCube} and \Cref{mainTorus} heavily rely on the combination of discrete and continuous ideas from \cite{bollobas_edge}. We do not define the edge-isoperimetric problem here, but note that continuous approximations have been similarly used for (vertex) isoperimetric problems in graphs, specifically the \emph{Cartesian} product of complete graphs by Harper \cite{harper1999l0}.

\section{The isoperimetric problem}

\begin{defn}
    A metric measure space is a tuple $(\Omega,\Sigma,\mu,d)$ where $(\Omega,d)$ is a metric space and $(\Omega,\Sigma,\mu)$ is a measure space where the $\sigma$-algebra $\Sigma$ contains all open subsets of $(\Omega,d)$.
\end{defn}

For a metric space $(\Omega,d)$ and a subset $A\subseteq\Omega$, define
\[
    N_\Omega[A]\eqdef\{\omega\in\Omega: d(\omega,a)<1\text{ for some }a\in A\}.
\]
Observe that $N_\Omega[A]$ is always an open set for any subset $A$ since it is the union of open balls of radius $1$.
It additionally has the property that
\begin{equation}\label{closure}
    N_\Omega[\overline{A}]=N_\Omega[A]
\end{equation}
where $\overline{A}$ denotes the closure of $A$.
\medskip

\begin{defn}[Isoperimetric function and iso-tightness]
    Let $\Omega=(\Omega,\Sigma,\mu,d)$ be a metric measure space.
    The \emph{isoperimetric function} of $\Omega$ is the function $\iso_\Omega\colon[0,\mu(\Omega)]\to[0,\mu(\Omega)]$ defined by
    \[
        \iso_\Omega(\alpha)\eqdef\inf\bigl\{\mu\bigl(N_\Omega[A]\bigr):\mu(A)\geq\alpha\bigr\}.
    \]

    We say that $\Omega$ is \emph{iso-tight} if $\iso_\Omega(\alpha)$ is achieved for every $\alpha\in[0,\mu(\Omega)]$.
    If $\Omega$ is iso-tight, then we define
    \[
        \Iso_\Omega(\alpha)\eqdef\bigl\{A\subseteq\Omega:\mu(A)\geq\alpha\text{ and } \mu\bigl(N_\Omega[A]\bigr)=\iso_\Omega(\alpha)\bigr\}.
    \]
\end{defn}
We note that $\iso_\Omega$ is (weakly) increasing and that $\iso_\Omega(\alpha)\geq\alpha$.
\medskip

Recall that a measure space $(\Omega,\Sigma,\mu)$ is said to be \emph{complete} if whenever $N\in\Sigma$ has $\mu(N)=0$ and $S\subseteq N$, then also $S\in\Sigma$ and $\mu(S)=0$.
Lebesgue's extension theorem implies that for any measure space $\Omega=(\Omega,\Sigma,\mu)$, there is a (not necessarily unique) complete measure space $\overline{\Omega}=(\Omega,\overline{\Sigma},\overline{\mu})$ with $\overline{\Sigma}\supseteq\Sigma$ and $\overline{\mu}(S)=\mu(S)$ for each $S\in\Sigma$.
It is not difficult to prove that if $\overline\Omega$ is any completion of $\Omega$, then $\iso_{\overline\Omega}=\iso_\Omega$.
The reasoning behind this is that $S\in\overline\Sigma$ if and only if there are $A,Z\in\Sigma$ for which $A\subseteq S\subseteq A\cup Z$ and $\mu(Z)=0$.

In this paper, we will be primarily concerned with products of metric measure spaces.
The above paragraph justifies that we need not distinguish between the true product and the completion thereof.

\section{Preliminaries and strong products of spaces}

A \emph{pre-order} on a set $\Omega$ is a relation $\preceq$ which is both reflexive and transitive.
Pre-orders are very similar to partial-orders, except they do not require anti-symmetry; that is, we can have both $x\preceq y$ and $y\preceq x$, yet $x\neq y$.
We say that $(\Omega,\preceq)$ is a \emph{chain} if for every $x,y\in\Omega$, either $x\preceq y$ or $y\preceq x$ (or both).
\medskip

Recall that a measure space $(\Omega,\Sigma,\mu)$ is \emph{$\sigma$-finite} if we can write $\Omega$ as the countable union of sets of finite measure.

Recall that a metric space is \emph{separable} if it contains a countable, dense subset.
Additionally, a metric space is \emph{complete} if every Cauchy sequence converges.
A metric space which is both separable and complete will be called \emph{Polish}.

At a few places throughout the paper, we will need to make use of the following ``$\tau$-additivity'' property of $\sigma$-finite, Polish metric measure spaces.

\begin{theorem}\label[theorem]{tauAdditive}
    Let $\Omega=(\Omega,\Sigma,\mu,d)$ be a $\sigma$-finite, Polish metric measure space.
    If $\mathcal A$ is a family of open subsets of $\Omega$ and $(\mathcal A,\subseteq)$ is a chain, then
    \[
        \mu\bigl(\bigcup\mathcal A\bigr)=\sup_{A\in\mathcal A}\mu(A).
    \]
\end{theorem}
\begin{proof}
    It is obvious that $\mu\bigl(\bigcup\mathcal A\bigr)\geq\sup_{A\in\mathcal A}\mu(A)$ and so we need only prove the reverse inequality.

    Ulam's theorem {\cite[Theorem 7.1.4]{dudley_book}} tells us that if $\Omega$ is a Polish metric \emph{probability} space, then for every measurable $S\in\Sigma$ and every $\epsilon>0$, there exists a compact $K\subseteq S$ with $\mu(S\setminus K)<\epsilon$.
    This immediately implies that if $\Omega$ is any $\sigma$-finite, Polish metric measure space, then
    \[
        \mu(S)=\sup\bigl\{\mu(K):K\subseteq S,\ K\text{ compact}\bigr\},\qquad\text{for every }S\in\Sigma.
    \]

    Set $S=\bigcup\mathcal A$ and consider any compact $K\subseteq S$.
    Of course, $\mathcal A$ is an open cover of $K$ and so there exists some finite sub-cover.
    Since $(\mathcal A,\subseteq)$ is a chain, this implies that $K\subseteq A$ for some $A\in\mathcal A$.
    In particular, $\mu(K)\leq\sup_{A\in\mathcal A}\mu(A)$, which concludes the proof.
\end{proof}

Fix measure spaces $\Omega_i=(\Omega_i,\Sigma_i,\mu_i)$ for $i\in\{1,2\}$.
Define $\Sigma_1\boxtimes\Sigma_2$ to be the smallest $\sigma$-algebra on $\Omega_1\times\Omega_2$ which contains the Cartesian products $\bigl\{S_1\times S_2:(S_1,S_2)\in\Sigma_1\times\Sigma_2\bigr\}$.
If each $\Omega_i$ is $\sigma$-finite, then it is well-known that there is a unique measure $\mu$ on $\Sigma_1\boxtimes\Sigma_2$ which satisfies
\[
    \mu(S_1\times S_2)=\mu_1(S_1)\cdot\mu_2(S_2)\qquad\text{for all }S_1\in\Sigma_1,\ S_2\in\Sigma_2.
\]
Call this unique measure $\mu=\mu_1\boxtimes\mu_2$.
\medskip

Consider a set $A\subseteq \Omega_1\times\Omega_2$.
For each $\omega_1\in\Omega_1$ and each $\omega_2\in\Omega_2$, define the slices
\begin{align*}
    \slicex{A}{\omega_2} &\eqdef\{x\in\Omega_1:(x,\omega_2)\in A\}\\
    \slicey{A}{\omega_1} &\eqdef\{y\in\Omega_2:(\omega_1,y)\in A\}.
\end{align*}
Observe that $x\in \slicex Ay\iff y\in \slicey Ax$.

\begin{theorem}[Fubini's Theorem]\label[theorem]{fubini}
    For $i\in\{1,2\}$, let $\Omega_i=(\Omega_i,\Sigma_i,\mu_i)$ be a $\sigma$-finite measure space.
    For every $A\in\Sigma_1\boxtimes\Sigma_2$,
    \begin{itemize}
        \item $\slicex A{\omega_2}\in\Sigma_1$ for every $\omega_2\in\Omega_2$ and
            \[
                (\mu_1\boxtimes\mu_2)(A)=\int_{\omega_2\in\Omega_2}\mu_1\bigl(\slicex{A}{\omega_2}\bigr)\ d\mu_2(\omega_2).
            \]
        \item $\slicey A{\omega_1}\in\Sigma_2$ for every $\omega_1\in\Omega_1$ and
            \[
                (\mu_1\boxtimes\mu_2)(A)=\int_{\omega_1\in\Omega_1}\mu_2\bigl(\slicey{A}{\omega_1}\bigr)\ d\mu_1(\omega_1).
            \]
    \end{itemize}
\end{theorem}

Given metric spaces $\Omega_i=(\Omega_i,d_i)$ for $i\in\{1,2\}$, we define a new metric space $\Omega_1\boxtimes\Omega_2$, which has ground-set $\Omega_1\times\Omega_2$ and metric $d_1\boxtimes d_2$ defined by
\[
    (d_1\boxtimes d_2)\bigl((x_1,x_2),(y_1,y_2)\bigr)=\max\bigl\{ d_1(x_1,y_1),\ d_2(x_2,y_2)\bigr\}.
\]
The metric $d_1\boxtimes d_2$ is analogous to the $\ell_\infty$ metric.
\medskip

\begin{defn}[Strong product of spaces]
For $\sigma$-finite metric measure spaces $\Omega_i=(\Omega_i,\Sigma_i,\mu_i,d_i)$ for $i\in\{1,2\}$, we define the \emph{strong product} to be
\[
    \Omega_1\boxtimes\Omega_2\eqdef\bigl(\Omega_1\times\Omega_2,\ \Sigma_1\boxtimes\Sigma_2,\ \mu_1\boxtimes\mu_2,\ d_1\boxtimes d_2\bigr).
\]
\end{defn}
It is easy to verify (see, for instance, {\cite[Proposition 3.3]{kechris_book}}) that $\Omega_1\boxtimes\Omega_2$ is Polish if both $\Omega_1$ and $\Omega_2$ are Polish.

\section{Compressions}

We open this section with the key property of $\boxtimes$ that makes its isoperimetric function behave well.

\begin{prop}\label[prop]{fiberNeighborhood}
    Suppose that $\Omega_i=(\Omega_i,d_i)$ are metric spaces for $i\in\{1,2\}$.
    For any $A\subseteq\Omega_1\times\Omega_2$ and any $\omega_1\in\Omega_1$,
    \[
        \slicey{N_{\Omega_1\boxtimes\Omega_2}[A]}{\omega_1}=\bigcup_{x\in N_{\Omega_1}[\omega_1]}N_{\Omega_2}\bigl[\slicey Ax\bigr].
    \]
\end{prop}
\begin{proof}
    By definition, $y\in\slicey{N_{\Omega_1\boxtimes\Omega_2}[A]}{\omega_1}$ if and only if $(\omega_1,y)\in N_{\Omega_1\boxtimes\Omega_2}[A]$.
    The latter happens if and only if there is some $x\in N_{\Omega_1}[\omega_1]$ for which $y\in N_{\Omega_2}[\slicey Ax]$, which concludes the proof.
\end{proof}

\begin{defn}[Iso-nesting]\label[defn]{defn:isoNesting}
    Let $\Omega=(\Omega,\Sigma,\mu,d)$ be a metric measure space.
    A function $F\colon\Sigma\to\Sigma$ is called an \emph{iso-nesting} of $\Omega$ if the following properties hold for every $A,B\in\Sigma$:
    \begin{itemize}
        \item $\mu\bigl(F(A)\bigr)\geq\mu(A)$, and
        \item $\mu\bigl(N\bigl[F(A)\bigr]\bigr)\leq\mu\bigl(N[A]\bigr)$, and
        \item $\mu(A)\leq\mu(B)\implies F(A)\subseteq F(B)$.
    \end{itemize}

\end{defn}
Note that in particular, the third condition implies that $\{F(A): A \in \Sigma\}$ forms a chain, which we interpret as the ``optimal'' sets with the same measure as $A$.
\begin{defn}[Compression]\label[defn]{defn:compression}
    For a function $F\colon 2^{\Omega_2}\to 2^{\Omega_2}$, the compression operator $C_F\colon 2^{\Omega_1\times\Omega_2}\to 2^{\Omega_1\times\Omega_2}$ is defined by
    \[
        C_F(A)\eqdef\bigcup_{\omega_1\in\Omega_1}\{\omega_1\}\times F\bigl(A\{\omega_1,\cdot\}\bigr)
    \]
\end{defn}

Observe that for any $\omega_1\in\Omega_1$, we have
\begin{equation}\label{eqn:compressSlice}
    C_F(A)\{\omega_1,\cdot\}=F\bigl(A\{\omega_1,\cdot\}\bigr).
\end{equation}

\begin{theorem}\label[theorem]{compressionsAreBetter}
    For $i\in\{1,2\}$, let $\Omega_i=(\Omega_i,\Sigma_i,\mu_i,d_i)$ be a $\sigma$-finite, Polish metric measure space.
    If $F\colon\Sigma_2\to\Sigma_2$ is an iso-nesting of $\Omega_2$, then for any $A\in\Sigma_1\boxtimes\Sigma_2$,
    \begin{align*}
        (\mu_1\boxtimes\mu_2)\bigl(C_F(A)\bigr) & \geq (\mu_1\boxtimes\mu_2)(A),\qquad\text{and}\\
        (\mu_1\boxtimes\mu_2)\bigl(N_{\Omega_1\boxtimes \Omega_2}\bigl[C_F(A)\bigr]\bigr) &\leq (\mu_1\boxtimes\mu_2)\bigl(N_{\Omega_1\boxtimes\Omega_2}[A]\bigr).
    \end{align*}
\end{theorem}
In other words, compressions along an iso-nesting make sets larger and neighborhoods smaller.

\begin{proof}
    Since $\mu_2\bigl(F(Y)\bigr)\geq\mu_2(Y)$ for every $Y\in\Sigma_2$, we can apply Fubini's theorem and \cref{eqn:compressSlice} to see that
    \begin{align*}
    (\mu_1\boxtimes\mu_2)\bigl(C_F(A)\bigr) &=\int_{\omega_1\in\Omega_1}\mu_2\bigl(\slicey{C_F(A)}{\omega_1}\bigr)\ d\mu_1(\omega_1) =\int_{\omega_1\in\Omega_1}\mu_2\bigl(F\bigl(\slicey{A}{\omega_1}\bigr)\bigr)\ d\mu_1(\omega_1)\\
                                          &\geq\int_{\omega_1\in\Omega_1}\mu_2\bigl(\slicey{A}{\omega_1}\bigr)\ d\mu_1(\omega_1) =(\mu_1\boxtimes\mu_2)(A).
    \end{align*}

    We work similarly for the second part.
    First, consider any $\mathcal Y\subseteq\Sigma_2$.
    By definition, $F(\Sigma_2)$ is a chain and so $\{F(Y):Y\in\mathcal Y\}$ is also a chain.
    In particular, the family $\bigl\{N_{\Omega_2}[F(Y)]:Y\in\mathcal Y\bigr\}$ is a chain of open sets and so \Cref{tauAdditive} tells us that
    \[
        \mu_2\biggl(\bigcup_{Y\in\mathcal Y}N_{\Omega_2}[F(Y)]\biggr)=\sup_{Y\in\mathcal Y}\mu_2\bigl(N_{\Omega_2}[F(Y)]\bigr).
    \]
    In particular, for any $\omega_1\in\Omega_1$, we can use \Cref{fiberNeighborhood} and \cref{eqn:compressSlice} to bound
    \begin{align*}
        \mu_2\bigl(\slicey{N_{\Omega_1\boxtimes\Omega_2}[C_F(A)]}{\omega_1}\bigr) &= \mu_2\biggl(\bigcup_{x\in N_{\Omega_1}[\omega_1]} N_{\Omega_2}\bigl[\slicey{C_F(A)}{x}\bigr]\biggr) = \mu_2\biggl(\bigcup_{x\in N_{\Omega_1}[\omega_1]} N_{\Omega_2}\bigl[F\bigl(\slicey{A}{x}\bigr)\bigr]\biggr)\\
                                                                                &= \sup_{x\in N_{\Omega_1}[\omega_1]}\mu_2\bigl(N_{\Omega_2}\bigl[F\bigl(\slicey Ax\bigr)\bigr]\bigr) \leq \sup_{x\in N_{\Omega_1}[\omega_1]}\mu_2\bigl(N_{\Omega_2}\bigl[\slicey Ax\bigr]\bigr)\\
                                                                                &\leq \mu_2\biggl(\bigcup_{x\in N_{\Omega_1}[\omega_1]}N_{\Omega_2}\bigl[\slicey Ax\bigr]\biggr)=\mu_2\bigl(\slicey{N_{\Omega_1\boxtimes\Omega_2}[A]}{\omega_1}\bigr).
    \end{align*}
    Hence, applying Fubini's theorem yields
    \begin{align*}
        (\mu_1\boxtimes\mu_2)\bigl(N_{\Omega_1\boxtimes\Omega_2}\bigl[C_F(A)\bigr]\bigr) &= \int_{\omega_1\in\Omega_1}\mu_2\bigl(\slicey{N_{\Omega_1\boxtimes\Omega_2}[C_F(A)]}{\omega_1}\bigr)\ d\mu_1(\omega_1)\\
                                                                                     &\leq \int_{\omega_1\in\Omega_1}\mu_2\bigl(\slicey{N_{\Omega_1\boxtimes\Omega_2}[A]}{\omega_1}\bigr)\ d\mu_1(\omega_1)\\
                                                                                     &= (\mu_1\boxtimes\mu_2)\bigl(N_{\Omega_1\boxtimes\Omega_2}[A]\bigr).\qedhere
    \end{align*}
\end{proof}



\begin{defn}[Layered sets]
    Let $\Omega_1,\Omega_2$ be two sets and fix a family $\mathcal F\subseteq 2^{\Omega_2}$.
    A set $A\subseteq\Omega_1\times\Omega_2$ is said to be \emph{$\mathcal F$-layered} if $\slicey{A}{\omega_1}\in\mathcal F$ for every $\omega_1\in\Omega_1$.
\end{defn}

Note that if $A\in\Sigma_1\boxtimes\Sigma_2$ where $\Sigma_1,\Sigma_2$ are $\sigma$-algebras, then $A$ is $\Sigma_2$-layered due to Fubini's theorem.
Furthermore, observe that if $F\colon\Sigma_2\to\Sigma_2$ is any function, then $C_F(A)$ is $F(\Sigma_2)$-layered.
In particular, \Cref{compressionsAreBetter} says that, when computing $\iso_{\Omega_1\boxtimes\Omega_2}$ in the case that $\Omega_2$ has an iso-nesting $F$, then we need only consider subsets which are $F(\Sigma_2)$-layered.
\medskip

Next, fix a set $\Omega$ and let $\mathcal F\subseteq 2^{\Omega}$ be some family of subsets of $\Omega$.
Define the pre-order $\preceq_{\mathcal F}$ on $\Omega$ by $x\preceq_{\mathcal F}y$ if and only if
\[
    y\in S\implies x\in S\qquad\text{for all }S\in\mathcal F
\]
Note that if $y\in\Omega\setminus\bigcup\mathcal F$, then $x\preceq_{\mathcal F}y$ for every $x\in\Omega$.

\begin{prop}\label[prop]{chainYieldsChain}
    For any $\mathcal F\subseteq 2^{\Omega}$, $(\mathcal F,\subseteq)$ is a chain if and only if $(\Omega,\preceq_{\mathcal F})$ is also a chain.
\end{prop}
\begin{proof}
    Suppose first that $(\mathcal F,\subseteq)$ is not a chain.
    Then we can locate $A,B\in\mathcal F$ such that both $A\setminus B$ and $B\setminus A$ are non-empty.
    If $x\in A\setminus B$ and $y\in B\setminus A$, then neither $x\preceq_{\mathcal F}y$ nor $y\preceq_{\mathcal F}x$ and so $(\Omega,\preceq_{\mathcal F})$ is not a chain.
    \medskip

    Suppose next that $(\mathcal F,\subseteq)$ is a chain.
    Fix any $x,y\in\Omega$ and suppose that $x\not\preceq_{\mathcal F}y$.
    Then there is some $A\in\mathcal F$ for which $y\in A$ yet $x\notin A$.
    Now, fix any $S\in\mathcal F$ and suppose that $x\in S$.
    Of course, $S\not\subseteq A$ since $x\notin A$, so it must be the case that $A\subseteq S$ since $\mathcal F$ is a chain under set inclusion.
    Therefore $y\in S$ as well.
    Since $S\in\mathcal F$ was arbitrary, this implies that $y\preceq_{\mathcal F}x$.
    We conclude that $(\Omega,\preceq_{\mathcal F})$ is a chain.
\end{proof}

\begin{theorem}\label[theorem]{slices}
    Let $\Omega_i=(\Omega_i,\Sigma_i,\mu_i,d_i)$ be a $\sigma$-finite, Polish metric measure space for each $i\in\{1,2\}$.
    Fix $A\in\Sigma_1\boxtimes\Sigma_2$ and define the functions $u_A,n_A\colon\Omega_2\to[0,\mu_1(\Omega_1)]$ by
    \[
        u_A(\omega_2)\eqdef\mu_1\bigl(A\{\cdot,\omega_2\}\bigr),\qquad\text{and}\qquad n_A(\omega_2)\eqdef\mu_1\bigl(N_{\Omega_1}[A\{\cdot,\omega_2\}]\bigr).
    \]
    Fix also a family $\mathcal F\subseteq 2^{\Omega_2}$.
    The following hold:
    \begin{enumerate}
        \item $\displaystyle(\mu_1\boxtimes\mu_2)(A)=\int_{\omega_2\in\Omega_2}u_A(\omega_2)\ d\mu_2(\omega_2)$.
        \item $n_A(\omega_2)\geq \iso_{\Omega_1}\bigl(u_A(\omega_2)\bigr)$ for all $\omega_2\in\Omega_2$.
        \item If $A$ is $\mathcal F$-layered and $x\preceq_{\mathcal F} y$, then $u_A(x)\geq u_A(y)$ and $n_A(x)\geq n_A(y)$.
        \item If $(\mathcal F,\subseteq)$ is a chain and $A$ is $\mathcal F$-layered, then
            \[
                (\mu_1\boxtimes\mu_2)\bigl(N_{\Omega_1\boxtimes\Omega_2}[A]\bigr)=\int_{\omega_2\in\Omega_2}\biggl(\sup_{y\in N_{\Omega_2}[\omega_2]} n_A(y)\biggr)\ d\mu_2(\omega_2).
            \]
    \end{enumerate}
\end{theorem}
\begin{proof}
    The first item follows from Fubini's theorem and the second item follows from definition.
    We thus focus only on the third and fourth items.
    As such, suppose that $A$ is $\mathcal F$-layered.
    \medskip

    Fix any $x\preceq_{\mathcal F} y\in\Omega_2$; we claim that $\slicex Ax\supseteq\slicex Ay$.
    To see this, fix any $\omega\in\slicex Ay$; then $y\in\slicey A\omega$.
    Now, $\slicey A\omega\in\mathcal F$ since $A$ is $\mathcal F$-layered and so $x\in\slicey A\omega$ as well since $x\preceq_{\mathcal F} y$ by assumption.
    Therefore, $\omega\in\slicex Ax$ as needed.
    \medskip

    Therefore, if $x\preceq_{\mathcal F}y$, it is obvious that $u_A(x)\geq u_A(y)$ and $n_A(x)\geq n_A(y)$ since $\slicex Ax\supseteq\slicex Ay$.
    \medskip

    We finally establish the fourth item.
    To begin, \Cref{chainYieldsChain} tells us that $(\Omega_2,\preceq_{\mathcal F})$ is a chain since $\mathcal F$ is a chain.
    In particular, due to the work above, this implies that the family $\{\slicex{A}{\omega_2}:\omega_2\in\Omega_2\}$ is a chain.
    Therefore, for any $Y\subseteq\Omega_2$, $\{\slicex Ay:y\in Y\}$ is a chain and so the family $\{N_{\Omega_1}[\slicex Ay]:y\in Y\}$ is a chain of open sets.
    As such, for any non-empty $Y\subseteq\Omega_2$, \Cref{tauAdditive} tells us that
    \[
        \mu_1\biggl(\bigcup_{y\in Y}N_{\Omega_1}[\slicex Ay]\biggr)=\sup_{y\in Y}\mu_1\bigl(N_{\Omega_1}[\slicex Ay]\bigr)=\sup_{y\in Y} n_A(y).
    \]
    Therefore, Fubini's theorem along with \Cref{fiberNeighborhood} yields
    \begin{align*}
        (\mu_1\boxtimes\mu_2)\bigl(N_{\Omega_1\boxtimes\Omega_2}[A]\bigr) &= \int_{\omega_2\in\Omega_2}\mu_1\bigl(\slicex{N_{\Omega_1\boxtimes\Omega_2}[A]}{\omega_2}\bigr)\ d\mu_2(\omega_2)\\
                                                                      &= \int_{\omega_2\in\Omega_2}\mu_1\biggl(\bigcup_{y\in N_{\Omega_2}[\omega_2]} N_{\Omega_1}\bigl[\slicex Ay\bigr]\biggr)\ d\mu_2(\omega_2)\\
                                                                      &= \int_{\omega_2\in\Omega_2}\biggl(\sup_{y\in N_{\Omega_2}[\omega_2]}n_A(y)\biggr)\ d\mu_2(\omega_2).\qedhere
    \end{align*}
\end{proof}

\section{Cubes, cylinders and tori}

We next outline the two basic metric measure spaces we will be using.
\begin{enumerate}
    \item For $\delta\in(0,1]$, the interval $I(\delta)$ has ground-set $[0,1]$ with the standard Lebesgue measure $\lebesgue$ and metric $d(x,y)={1\over\delta}\abs{x-y}$.
    \item For $\delta\in(0,1]$, the circle $S(\delta)$ has ground-set $[0,1)$ with the standard Lebesgue measure $\lebesgue$ and metric $d(x,y)={1\over\delta}\min\bigl\{\abs{x-y},1-\abs{x-y}\bigr\}$.
\end{enumerate}

These two spaces are incredibly nice.
Firstly, they are both clearly metric probability spaces (and hence $\sigma$-finite) and are Polish.
Let $\Sigma$ denote the $\sigma$-algebra of Lebesgue-measurable subsets of $[0,1]$ and define the set of \emph{rooted intervals} to be
\[
    \rootedIntervals\eqdef\bigl\{[0,x):x\in[0,1]\bigr\}\cup\bigl\{[0,x]:x\in[0,1)\bigr\}\subseteq\Sigma.
\]
\begin{theorem}\label[theorem]{1d}
    Fix $X\in\bigl\{I(\delta),S(\delta)\bigr\}$.
    The function $R\colon\Sigma\to\rootedIntervals$ defined by $R(A)=\bigl[0,\lebesgue(A)\bigr)$ is an iso-nesting of $X$.
    Furthermore, $X$ is iso-tight and $R(A)\in\Iso_X\bigl(\lebesgue(A)\bigr)$ for each $A\in\Sigma$.
\end{theorem}
\begin{proof}
    We clearly have $\lebesgue\bigl(R(A)\bigr)=\lebesgue(A)$ for any $A\in\Sigma$; we need to prove that $\lebesgue\bigl(N_X[R(A)]\bigr)\leq\lebesgue\bigl(N_X[A]\bigr)$.

    We begin by focusing on $S(\delta)$.
    Here, $\lebesgue\bigl(N\bigl[R(A)\bigr]\bigr)=\min\{1,\lebesgue(A)+2\delta\}$.

    Suppose that $N_{S(\delta)}[A]\neq [0,1)$.
    Using the cyclic symmetry of $S(\delta)$, we can express $N_{S(\delta)}[A]=I_1\sqcup\dots\sqcup I_n$ where each $I_i$ is an interval in $[0,1)$ and $\sup I_i\leq\inf I_{i+1}$ for each $i\in[n-1]$.
    Furthermore, we may suppose that $0\notin I_1$ since $N_{S(\delta)}[A]\neq [0,1)$.
    With this consideration, we find that for each $i\in[n]$, $\inf I_i+\delta\leq\inf(A\cap I_i)\leq\sup(A\cap I_i)\leq \sup I_i-\delta$ and so
    \[
        \lebesgue\bigl(N_{S(\delta)}[A]\bigr)=\sum_{i=1}^n\lebesgue(I_i)\geq \sum_{i=1}^n\bigl(\lebesgue(A\cap I_i)+2\delta\bigr)=\lebesgue(A)+2n\delta\geq \min\{1,\lebesgue(A)+2\delta\}=\lebesgue\bigl(N_{S(\delta)}[A]\bigr).
    \]

    We now consider $I(\delta)$.
    Here, $\lebesgue\bigl(N_{I(\delta)}\bigl[R(A)\bigr]\bigr)=\min\{1,\lebesgue(A)+\delta\}$.
    Similarly to before, we can express $N_{I(\delta)}[A]=I_1\sqcup\dots\sqcup I_n$ where $I_i$ is an interval in $[0,1]$ and $\sup I_i\leq \inf_{i+1}$ for each $i\in[n-1]$.
    Now, for each $i\in\{1,\dots,n-1\}$, $\sup(A\cap I_i)\leq\sup I_i-\delta$ and for each $i\in\{2,\dots,n-1\}$, $\inf(A\cap I_i)\geq\inf I_i+\delta$.
    In particular, if $n\geq 2$, then
    \begin{align*}
        \lebesgue\bigl(N_{I(\delta)}[A]\bigr) &=\sum_{i=1}^n\lebesgue(I_i)\geq\bigl(\lebesgue(A\cap I_1)+\delta\bigr)+\sum_{i=2}^{n-1}\bigl(\lebesgue(A\cap I_i)+2\delta\bigr)+\bigl(\lebesgue(A\cap I_n)+\delta\bigr)\\
                                              &=\lebesgue(A)+2(n-1)\delta\geq\min\{1,\lebesgue(A)+\delta\}=\lebesgue\bigl(N_{I(\delta)}\bigl[R(A)\bigr]\bigr)
    \end{align*}
    We may thus suppose that $n=1$.
    Here, we have $\inf A-\inf I_1=\min\{\delta,\inf A\}$ and $\sup I_1-\sup A=\min\{\delta,1-\sup A\}$ and so
    \begin{align*}
        \lebesgue\bigl(N_{I(\delta)}[A]\bigr) &= \sup A-\inf A+\min\{\delta,1-\sup A\}+\min\{\delta,\inf A\}\\
                                              &\geq\lebesgue(A)+\min\{\delta,1-\sup A\}+\min\{\delta,\inf A\}.
    \end{align*}
    If either $\inf A\geq\delta$ or $1-\sup A\geq\delta$, then we are done.
    Otherwise, $\inf A<\delta$ and $1-\sup A<\delta$ and so
    \begin{align*}
        \lebesgue\bigl(N_{I(\delta)}[A]\bigr) &= \sup A-\inf A+\min\{\delta,1-\sup A\}+\min\{\delta,\inf A\}\\
                                              &= \sup A-\inf A+(1-\sup A)+\inf A=1,
    \end{align*}
    which concludes the proof.
\end{proof}


We can furthermore use \Cref{slices} to get a better understanding of this isoperimetric function.


\begin{theorem}\label[theorem]{cubeSlices}
    Suppose that $\Omega=(\Omega,\Sigma,\mu,d)$ is a $\sigma$-finite Polish metric measure space.
    Fix a $(\mu\boxtimes\lebesgue)$-measurable $A\subseteq\Omega\times[0,1)$ and define the functions $u_A,n_A\colon[0,1)\to[0,\mu(\Omega)]$ by
    \[
        u_A(t)\eqdef\mu\bigl(\slicex At\bigr),\qquad\text{and}\qquad n_A(t)\eqdef\mu\bigl(N_\Omega\bigl[\slicex At\bigr]\bigr).
    \]
    The following hold:
    \begin{enumerate}
        \item $\displaystyle(\mu\boxtimes\lebesgue)(A)=\int_0^1 u_A(t)\ dt$.
        \item $n_A(t)\geq\iso_\Omega\bigl(u_A(t)\bigr)$ for all $t\in[0,1)$.
        \item If $A$ is $\rootedIntervals$-layered, then both $u_A$ and $n_A$ are (weakly) decreasing.
        \item If $A$ is $\rootedIntervals$-layered, then
            \begin{align*}
                (\mu\boxtimes\lebesgue)\bigl(N_{\Omega\boxtimes I(\delta)}[A]\bigr) &= \phantom{2}\delta\cdot n_A(0)+\int_0^{1-\phantom{2}\delta}n_A(t)\ dt,\quad\text{for }0<\delta\leq 1\phantom{/2}\qquad\text{and}\\
                (\mu\boxtimes\lebesgue)\bigl(N_{\Omega\boxtimes S(\delta)}[A]\bigr) &= 2\delta\cdot n_A(0)+\int_0^{1-2\delta}n_A(t)\ dt,\quad\text{for }0<\delta\leq 1/2.
            \end{align*}
    \end{enumerate}
\end{theorem}
\begin{proof}
    The first two items are restatements of the corresponding items in \Cref{slices}.

    The third item follows from the corresponding item in \Cref{slices} along with the observation that $x\preceq_\rootedIntervals y\iff x\leq y$.

    The fourth item requires some additional work.
    Fix $X\in\{I(\delta),S(\delta)\}$; we will later case on the value of $X$.
    Firstly, it is clear that $(\rootedIntervals,\subseteq)$ is a chain and so the fourth item of \Cref{slices} tells us that
    \[
        (\mu\boxtimes\lebesgue)\bigl(N_{\Omega\boxtimes X}[A]\bigr) = \int_0^1\biggl(\sup_{y\in N_X[t]} n_A(y)\biggr)\ dt
    \]

    Now, $n_A$ is a decreasing function and so $\sup_{y\in N_X[t]}n_A(y)\leq n_A\bigl(\inf N_X[t]\bigr)$.
    Define $T\subseteq[0,1)$ to be the set of all $t$ for which $\sup_{y\in N_X[t]}n_A(y)< n_A\bigl(\inf N_X[t]\bigr)$.
    Naturally, if $t\in T$, then $(\inf N_X[t])\notin N_X[t]$ and so $(\inf N_X[t])=t-\delta$.
    Next, observe that if $t\in T$, then $n_A$ \emph{cannot} be right-continuous at $(\inf N_X[t])$.
    Since $n_A$ is decreasing, it has only countably-many discontinuities, and so we conclude that $T$ is countable.
    In particular, $\sup_{y\in N_X[t]}n_A(y)=n_A\bigl(\inf N_X[t]\bigr)$ for $\lebesgue$-a.e.\ $t\in[0,1)$ and so
    \[
        (\mu\boxtimes\lebesgue)\bigl(N_{\Omega\boxtimes X}[A]\bigr) = \int_0^1 n_A\bigl(\inf N_X[t]\bigr)\ dt.
    \]

    We now break into cases depending on $X$.
    If $X=I(\delta)$, then $\inf N_{I(\delta)}[t]=\max\{0,t-\delta\}$ and so
    \begin{align*}
        (\mu\boxtimes\lebesgue)\bigl(N_{\Omega\boxtimes I(\delta)}[A]\bigr) &= \int_0^1 n_A\bigl(\max\{0,t-\delta\}\bigr)\ dt = \delta\cdot n_A(0)+\int_\delta^1 n_A(t-\delta)\ dt\\
                                                                        &= \delta\cdot n_A(0)+\int_0^{1-\delta}n_A(t)\ dt.
    \end{align*}
    If $X=S(\delta)$, then since $N_X[t]$ is open,
    \[
        \inf N_{S(\delta)}[t] = \begin{cases}
            0, & \text{if }t+\delta>1,\\
            \max\{0,t-\delta\}, & \text{otherwise.}
        \end{cases}
    \]
    Therefore,
    \begin{align*}
        (\mu\boxtimes\lebesgue)\bigl(N_{\Omega\boxtimes S(\delta)}[A]\bigr) &= \delta\cdot n_A(0)+\int_0^{1-\delta}n_A\bigl(\max\{0,t-\delta\}\bigr)\ dt = 2\delta\cdot n_A(0)+\int_\delta^{1-\delta}n_A(t-\delta)\ dt\\
                                                                        &= 2\delta\cdot n_A(0)+\int_0^{1-2\delta}n_A(t)\ dt.\qedhere
    \end{align*}
\end{proof}

One particularly nice consequence of the above theorem is that intervals and circles are essentially identical:

\begin{corollary}\label[corollary]{circle=interval1}
    If $\Omega$ is a $\sigma$-finite Polish metric measure space, then $\iso_{\Omega\boxtimes I(\delta)}=\iso_{\Omega\boxtimes S(\delta/2)}$ for any $0<\delta\leq 1$.
\end{corollary}
\begin{proof}
    Let $A\subseteq\Omega\times[0,1)$ be $\rootedIntervals$-layered.
    Then, according to part 4 of \Cref{cubeSlices}, we know that $(\mu\boxtimes\lebesgue)\bigl(N_{\Omega\boxtimes I(\delta)}[A]\bigr)=(\mu\boxtimes\lebesgue)\bigl(N_{\Omega\boxtimes S(\delta/2)}[A]\bigr)$.
    Due to \Cref{compressionsAreBetter,1d}, we know that both $\iso_{\Omega\boxtimes I(\delta)}$ and $\iso_{\Omega\boxtimes S(\delta/2)}$ can be computed by considering only $\rootedIntervals$-layered sets.
    Thus, the claim follows.
\end{proof}

We will need the following consequence of this fact: 
\begin{corollary}\label[corollary]{circle=interval}
    Fix $0<\delta_1,\dots,\delta_n\leq 1$.
    If $X_i\in\{I(\delta_i),S(\delta_i/2)\}$ for all $i$, then
    \[
        \iso_{X_1\boxtimes\dots\boxtimes X_n}=\iso_{I(\delta_1)\boxtimes\dots\boxtimes I(\delta_n)}.
    \]
\end{corollary}
\begin{proof}
    We prove this by induction on the number of $i$'s for which $X_i=S(\delta_i/2)$.
    If there is no such $i$, then the claim is trivial, so suppose that there is at least one such $i$.
    Without loss of generality, $X_n=S(\delta_n/2)$ since the strong product is commutative.
    But then we can apply \Cref{circle=interval1} with $\Omega=X_1\boxtimes\dots\boxtimes X_{n-1}$ to conclude that $\iso_{X_1\boxtimes\dots\boxtimes X_n}=\iso_{X_1\boxtimes\dots\boxtimes X_{n-1}\boxtimes I(\delta_n)}$.
    Of course, from the inductive hypothesis, we know that $\iso_{X_1\boxtimes\dots\boxtimes X_{n-1}\boxtimes I(\delta_n)}=\iso_{I(\delta_1)\boxtimes\dots\boxtimes I(\delta_n)}$ and so the claim follows.
\end{proof}

Therefore, we will work solely with products of intervals going forward.
\medskip

The remainder of this section is dedicated to the proof of the following two theorems which are central in our proof of the isoperimetric inequality.

\begin{theorem}\label[theorem]{isotightClassification}
    Let $\Omega$ be a Polish metric probability space and fix $\delta\in(0,1)$.
    Suppose
    that $\iso_\Omega$ is continuous on $(0,1]$ and concave on $[0,1]$.
    Then,
    for any $\alpha\in(0,1]$, there exists some $\xi^*\in[0,1-\delta]$ satisfying:
    \begin{itemize}
        \item If $\xi^*\geq\alpha$, then for any $S\in\Iso_\Omega\bigl({\alpha\over\xi^*}\bigr)$, we have $S\times[0,\xi^*)\in\Iso_{\Omega\boxtimes I(\delta)}(\alpha)$.
        \item If $\xi^*\leq\alpha$, then for any $S\in\Iso_\Omega\bigl({\alpha-\xi^*\over 1-\xi^*}\bigr)$, we have $(S\times[0,1])\cup(\Omega\times[0,\xi^*))\in\Iso_{\Omega\boxtimes I(\delta)}(\alpha)$.
    \end{itemize}
    In particular, if $\Omega$ is iso-tight, then so is $\Omega \boxtimes I(\delta)$.
\end{theorem}

\begin{theorem}\label[theorem]{optimalStep}
    Let $\Omega$ be a Polish metric probability space and fix $\delta\in(0,1)$.
    Suppose that $\Omega$ is iso-tight, and that $\iso_\Omega$ is continuous on $(0,1]$ and concave on $[0,1]$.
    Suppose that $A\subseteq\Omega\times[0,1)$ is $\rootedIntervals$-layered and that there exist numbers $0<x<y<1$ and $\xi\in(0,1-\delta)$ such that
    \[
        u_A(t)=\begin{cases}
            y, & \text{if }t<\xi,\\
            x, & \text{if }t>\xi.
        \end{cases}
    \]
    If $A\in\Iso_\Omega(\alpha)$, then
    \[
        (1-\delta-\xi)\iso_\Omega(t)+(\delta+\xi)\iso_\Omega\biggl({\alpha-(1-\xi)t\over\xi}\biggr)=\iso_\Omega(\alpha),
    \]
    for all $\max\bigl\{0,{\alpha-\xi\over 1-\xi}\bigr\}\leq t\leq\alpha$.
    Furthermore, note that $\max\bigl\{0,{\alpha-\xi\over 1-\xi}\bigr\}\leq x<\alpha$ here.
\end{theorem}

Because $\max\bigl\{0,{\alpha-\xi\over 1-\xi}\bigr\}\leq x<\alpha$, the above theorem applies for $x\leq t\leq\alpha$.
This is the range in which we will explicitly apply the theorem going forward.

\subsection{Proofs of \Cref{isotightClassification,optimalStep}}

Throughout this section, we suppose that $\Omega=(\Omega,\Sigma,\mu,d)$ is a Polish metric probability space and $\delta\in(0,1)$ is fixed.
We furthermore assume that $\Omega$ is iso-tight, and that $\iso_\Omega$ is continuous on $(0,1]$ and concave on $[0,1]$.
We also set $\gamma\eqdef 1-\delta$ for notational ease.
\medskip

The next three statements outline fairly straightforward inequalities pertaining to concave functions.
\begin{prop}\label[prop]{concaveExpectation}
    Fix real numbers $a\leq b$ and suppose that $X$ is any random variable with $a\leq X\leq b$.
    We can therefore write $\E X=\lambda a+(1-\lambda)b$ for some $\lambda\in[0,1]$.
    For any concave function $f\colon[a,b]\to\R$,
    \[
        \E f(X)\geq \lambda f(a)+(1-\lambda)f(b).
    \]
\end{prop}
\begin{proof}
    If $a=b$, then the claim is trivial, so suppose that $a<b$.
    Define the random variable $Y$ by $X=Ya+(1-Y)b$ and note that $0\leq Y\leq 1$ since $a\leq X\leq b$.
    Additionally, $\E X=\lambda a+(1-\lambda)b$ and $a<b$ together imply that $\E Y=\lambda$.
    Therefore, since $f$ is concave,
    \[
        \E f(X)=\E f\bigl(Ya+(1-Y)b\bigr)\geq \E\bigl[Yf(a)+(1-Y)f(b)\bigr]=\lambda f(a)+(1-\lambda)f(b).\qedhere
    \]
\end{proof}

\begin{prop}\label[prop]{nonConstantConcave}
    Fix $a\leq b\leq c$.
    For any concave function $f\colon[a,c]\to\R$, we have $f(b)\geq\min\{f(a),f(c)\}$.
    Furthermore, if $a<b<c$, then $f(b)=\min\{f(a),f(c)\}$ if and only if $f$ is constant.
\end{prop}
\begin{proof}
    Writing $b=\lambda a+(1-\lambda)c$ for $\lambda\in[0,1]$, we have
    \[
        f(b)\geq \lambda f(a)+(1-\lambda)f(c)\geq\min\{f(a),f(c)\}.
    \]
    Next, suppose that $f$ is not constant and $b\in(a,c)$ and $f(b)=\min\{f(a),f(c)\}$.
    Since $f(b)\geq\lambda f(a)+(1-\lambda)f(c)$ and $\lambda\in(0,1)$, we find that we must have $f(a)=f(b)=f(c)$.
    However, since $f$ is not constant, there is some $d\in(a,c)$ with $f(d) > f(a)=f(b)=f(c)$; without loss of generality, we may suppose that $d\in(b,c)$.
    But then, the point $(b,f(b))$ lies strictly below the line connecting $(a,f(a))$ to $(d,f(d))$, contradicting the fact that $f$ is concave.
\end{proof}

\begin{prop}\label[prop]{concaveBoundary}
    Fix a concave function $f\colon[0,1]\to\R$, a number $\alpha\in[0,1]$ and a number $\lambda\in(0,1)$.
    For any $A,B\geq 0$ and any $0\leq x\leq y\leq 1$ satisfying $\alpha=\lambda y+(1-\lambda)x$,
    \[
        A f(x)+B f(y)\geq\min\begin{cases}
            (A+B)f(\alpha),\\
            \begin{cases}
                A f(0)+B f\bigl({\alpha\over\lambda}\bigr), & \text{if }\alpha\leq\lambda,\\
                A f\bigl({\alpha-\lambda\over 1-\lambda}\bigr)+B f(1), & \text{if }\alpha\geq\lambda.
            \end{cases}
        \end{cases}
    \]
    If $0<x<y<1$ and equality holds, then the function
    \[
        F(t)=A f(t)+Bf\biggl({\alpha-(1-\lambda)t\over\lambda}\biggr),
    \]
    is constant on the interval $\max\bigl\{0,{\alpha-\lambda\over 1-\lambda}\bigr\}\leq t\leq \alpha$.
\end{prop}
\begin{proof}
    Since $\lambda\in(0,1)$, we can write
    \[
        y={\alpha-(1-\lambda)x\over\lambda}.
    \]
    Therefore, the bounds $0\leq x\leq y\leq 1$ can be rewritten as
    \[
        \max\biggl\{0,{\alpha-\lambda\over 1-\lambda}\biggr\}\leq x\leq \alpha.
    \]
    Now, define the function
    \[
        F(t)\eqdef A f(t)+B f\biggl({\alpha-(1-\lambda)t\over\lambda}\biggr),
    \]
    so $F(x)=Af(x)+Bf(y)$.
    Since $f$ is concave, $F$ is also concave.
    The claim now follows by applying \Cref{nonConstantConcave} to $F$.
\end{proof}


For a function $f\colon\R\to\R$ and a number $t\in\R$, define
\[
    f(t^+)\eqdef\lim_{s\to t^+}f(s)\qquad\text{and}\qquad f(t^-)\eqdef\lim_{s\to t^-}f(s),
\]
should these numbers exist.
Note that if $f$ is monotone, then these parameters are always defined.
\medskip

Recall that we defined $\gamma\eqdef 1-\delta$ for notational ease.
\begin{defn}
    For $\alpha\in[0,1]$, define the function $B_\alpha(t)\colon[0,\gamma]\to[0,1]$ by
    \[
        B_\alpha(t)\eqdef\begin{cases}
            \iso_\Omega(\alpha), & \text{if }t=0,\\
            (1-\gamma+t)+(\gamma-t)\iso_\Omega\bigl({\alpha-t\over 1-t}\bigr), & \text{if }0<t\leq\alpha,\\
            (1-\gamma+t)\iso_\Omega\bigl({\alpha\over t}\bigr), & \text{if }\alpha<t\leq\gamma.
        \end{cases}
    \]
\end{defn}
Note that $B_\alpha$ is continuous at $\alpha$, but not at 0 (unless $\iso_\Omega(\alpha)=1$).
\begin{defn}
    For a weakly decreasing function $u\colon[0,\gamma]\to\R$, define the number $\xi(u)\in(0,\gamma)$ by
    \[
        \int_0^{\gamma}u(t)\ dt=\xi(u)\cdot u(0^+)+\bigl(\gamma-\xi(u)\bigr)\cdot u(\gamma^-)
    \]
    If there are multiple choices for $\xi(u)$, select one arbitrarily.
    Note that $\xi(u)$ always exists since $u(0^+)\geq u(t)\geq u(\gamma^-)$ for all $t\in(0,\gamma)$.
\end{defn}

\begin{theorem}\label[theorem]{mainBound}
    Suppose that $A\subseteq\Omega\times[0,1)$ is $\rootedIntervals$-layered and set $\alpha=(\mu\boxtimes\lebesgue)(A)$.
    Then,
    \[
        (\mu\boxtimes\lebesgue)\bigl(N_{\Omega\boxtimes I(\delta)}[A]\bigr)\geq\min_{t\in[0,\gamma]}B_\alpha(t).
    \]
    Furthermore, suppose that equality holds and set $\xi=\xi(u_A)$ where $u_A(t)=\mu\bigl(\slicex{A}{t}\bigr)$.
    If $0<{\alpha-\xi\cdot u_A(0^+)\over 1-\xi}<u_A(0^+)<1$, then
    \[
        (\gamma-\xi)\iso_\Omega(t)+(1-\gamma+\xi)\iso_\Omega\biggl({\alpha-(1-\xi)t\over \xi}\biggr)=\iso_\Omega(\alpha),
    \]
    for all $\max\bigl\{0,{\alpha-\xi\over 1-\xi}\bigr\}\leq t\leq\alpha$.
\end{theorem}

In particular, $\iso_{\Omega\boxtimes I(\delta)}(\alpha)\geq\min_{t\in[0,\gamma]}B_\alpha(t)$.
We will see later that every $B_\alpha(t)$ is the measure of some $N_{\Omega \boxtimes I(\delta)}[A]$, so that equality actually holds here.

\begin{proof}[Proof of \Cref{mainBound}]
    Consider selecting a uniformly random number from the interval $(0,\gamma)$, which is a valid probability space since $\gamma>0$.
    If $X$ is any random variable on this space, then $\E X={1\over\gamma}\int_0^\gamma X(t)\ dt$.
    Considering $u_A$ as a random variable, we have $\E u_A={\xi\over\gamma} u_A(0^+)+(1-{\xi\over\gamma})u_A(\gamma^-)$.
    Since $u_A$ is weakly decreasing (\Cref{cubeSlices}), we have $u_A(\gamma^-)\leq u_A(t)\leq u_A(0^+)$ for all $t\in(0,\gamma)$ and so we may apply \Cref{concaveExpectation}
    to bound
    \[
        \int_0^\gamma \iso_\Omega\bigl(u_A(t)\bigr)\ dt = \gamma\E\iso_\Omega\bigl(u_A\bigr)\geq\xi\cdot \iso_\Omega\bigl(u_A(0^+)\bigr)+(\gamma-\xi)\cdot\iso_\Omega\bigl(u_A(\gamma^-)\bigr),
    \]
    since $\iso_\Omega$ is assumed to be concave.
    In particular, we may apply \Cref{cubeSlices} to bound
    \begin{align}
        (\mu\boxtimes\lebesgue)\bigl(N_{\Omega\boxtimes I(\delta)}[A]\bigr) &=
        (1-\gamma) \cdot n_A(0) + \int_0^\gamma n_A(t) dt \nonumber
        \\
        & \geq
        (1-\gamma)\iso_\Omega\bigl(u_A(0)\bigr)+\int_0^\gamma\iso_\Omega\bigl(u_A(t)\bigr)\ dt\nonumber \\
                                                                            &\geq (1-\gamma+\xi)\iso_\Omega\bigl(u_A(0^+)\bigr)+(\gamma-\xi)\iso_\Omega\bigl(u_A(\gamma^-)\bigr),\label{eqn:stepFunction}
    \end{align}
    where we additionally used the fact that $\iso_\Omega$ is increasing and $u_A$ is decreasing in the inequalities.

    Next, define the numbers
    \[
        y\eqdef u_A(0^+),\qquad\text{and}\qquad x\eqdef{\alpha-\xi\cdot u_A(0^+)\over 1-\xi}.
    \]
    By the definition of $\xi=\xi(u_A)$, we note that
    \[
        x={1\over 1-\xi}\biggl((\gamma-\xi)u_A(\gamma^-)+\int_\gamma^1 u_A(t)\ dt\biggr).
    \]
    From this, we immediately see that $x\geq 0$.
    Furthermore, since $u_A$ is decreasing we also have
    \[
        x\leq{1\over 1-\xi}\biggl((\gamma-\xi)u_A(\gamma^-)+\int_\gamma^1 u_A(\gamma^-)\ dt\biggr)=u_A(\gamma^-).
    \]
    Thus, since $\iso_\Omega$ is increasing, we can continue \cref{eqn:stepFunction} to bound
    \begin{equation}\label{eqn:betterStep}
        (\mu\boxtimes\lebesgue)\bigl(N_{\Omega\boxtimes I(\delta)}[A]\bigr)\geq (1-\gamma+\xi)\iso_\Omega(y)+(\gamma-\xi)\iso_\Omega(x).
    \end{equation}

    Next, we know that $0<\xi<\gamma<1$ and that $0\leq x\leq y\leq 1$.
    Furthermore, by construction, $\xi y+(1-\xi)x=\alpha$.
    We may therefore apply \Cref{concaveBoundary} with $\lambda=\xi$, $f=\iso_\Omega$, $A=\gamma-\xi$ and $B=1-\gamma+\xi$ to continue \cref{eqn:betterStep} and bound
    \begin{equation}\label{eqn:boundaryStep}
        (\mu\boxtimes\lebesgue)\bigl(N_{\Omega\boxtimes I(\delta)}[A]\bigr)\geq\min\begin{cases}
            \iso_\Omega(\alpha),\\
            \begin{cases}
                (1-\gamma+\xi)+(\gamma-\xi)\iso_\Omega\bigl({\alpha-\xi\over 1-\xi}\bigr), & \text{if }\xi\leq\alpha,\\
                (1-\gamma+\xi)\iso_\Omega\bigl({\alpha\over\xi}\bigr), & \text{if }\alpha\leq\xi,
            \end{cases}
        \end{cases}
    \end{equation}
    where we additionally used the fact that $\iso_\Omega(0)=0$ and $\iso_\Omega(1)=1$.
    According to \Cref{concaveBoundary}, if $0<x<y<1$ and equality holds in \cref{eqn:boundaryStep}, then
    \[
        (\gamma-\xi)\iso_\Omega(t)+(1-\gamma+\xi)\iso_\Omega\biggl({\alpha-(1-\xi)t\over\xi}\biggr)=\iso_\Omega(\alpha),
    \]
    for all $\max\bigl\{0,{\alpha-\xi\over 1-\xi}\bigr\}\leq t\leq\alpha$.
    Thus, in order to finish the proof, it suffices to show that
    \[
        \min_{t\in[0,\gamma]}B_\alpha(t)\leq\min\begin{cases}
            \iso_\Omega(\alpha),\\
            \begin{cases}
                (1-\gamma+\xi)+(\gamma-\xi)\iso_\Omega\bigl({\alpha-\xi\over 1-\xi}\bigr), & \text{if }\xi\leq\alpha,\\
                (1-\gamma+\xi)\iso_\Omega\bigl({\alpha\over\xi}\bigr), & \text{if }\alpha\leq\xi,
            \end{cases}
        \end{cases}
    \]
    which, by definition, amounts to showing that
    \[
        \min_{t\in[0,\gamma]}B_\alpha(t)\leq\min\bigl\{B_\alpha(0), B_\alpha(\xi)\bigr\},
    \]
    which, provided the minimum exists, is trivial since $\xi<\gamma$.
    Thus, we need only prove that $\min_{t\in[0,\gamma]}B_\alpha(t)$ actually exists.
    To see this, observe that $B_\alpha(t)$ is continuous for all $t\in(0,\gamma]$ since $\iso_\Omega$ is continuous on $(0,1]$; thus it suffices to show that $\lim_{t\to 0^+}B_\alpha(t)\geq B_\alpha(0)$.
    If $\alpha=0$, then $B_\alpha$ is identically 0 so this is trivial. Otherwise
    \[
        \lim_{t\to 0^+}B_\alpha(t)=\lim_{t\to 0^+}\biggl((1-\gamma+t)+(\gamma-t)\iso_\Omega\biggl({\alpha-t\over 1-t}\biggr)\biggr)=(1-\gamma)+\gamma\cdot\iso_\Omega(\alpha)\geq\iso_\Omega(\alpha)=B_\alpha(0).\qedhere
    \]
\end{proof}

We now prove \Cref{isotightClassification}.
\begin{proof}[Proof of \Cref{isotightClassification}]
    Pick $\xi^*\in[0,\gamma]$ so that $B_\alpha(\xi^*)=\min_{t\in[0,\gamma]}B_\alpha(t)$.

    If $\xi^*\geq\alpha$, then fix any $S\in\Iso_\Omega\bigl({\alpha\over\xi^*}\bigr)$ and set $A=S\times[0,\xi^*)$.
    Naturally $(\mu\boxtimes\lebesgue)(A)=\mu(S)\cdot\xi^*=\alpha$, and
    \[
        (\mu\boxtimes\lebesgue)\bigl(N_{\Omega\boxtimes I(\delta)}[A]\bigr)=(\delta+\xi^*)\mu\bigl(N_\Omega(S)\bigr)=(\delta+\xi^*)\iso_\Omega\biggl({\alpha\over\xi^*}\biggr)=B_\alpha(\xi^*).
    \]
    Thus, \Cref{mainBound} implies that $A\in\Iso_\Omega(\alpha)$ as claimed.

    If $\xi^*\leq\alpha$, then fix any $S\in\Iso_\Omega({\alpha-\xi^*\over 1-\xi^*})$ and set $A=\bigl(\Omega\times[0,\xi^*)\bigr)\sqcup\bigl(S\times[\xi^*,1)\bigr)$.
    Naturally, $(\mu\boxtimes\lebesgue)(A)=\xi^*+\mu(S)(1-\xi^*)=\alpha$.
    If $\xi^*=0$, then
    \[
        (\mu\boxtimes\lebesgue)\bigl(N_{\Omega\boxtimes I(\delta)}[A]\bigr)=\mu\bigl(N_\Omega[S]\bigr)=\iso_\Omega(\alpha)=B_\alpha(\xi^*).
    \]
    If $\xi^*>0$, then
    \begin{align*}
        (\mu\boxtimes\lebesgue)\bigl(N_{\Omega\boxtimes I(\delta)}[A]\bigr) &=(\delta+\xi^*)+(1-\delta-\xi^*)\mu\bigl(N_\Omega[S]\bigr)\\
                                                                            &=(\delta+\xi^*)+(1-\delta-\xi^*)\iso_\Omega\biggl({\alpha-\xi^*\over 1-\xi^*}\biggr)=B_\alpha(\xi^*).
    \end{align*}
    In either case, \Cref{mainBound} implies that $A\in\Iso_\Omega(\alpha)$ as claimed.
\end{proof}
An important consequence of the above proof is that $\iso_\Omega(\alpha)=\min_{t\in[0,\gamma]}B_\alpha(t)$.

Finally, we prove \Cref{optimalStep}.
\begin{proof}[Proof of \Cref{optimalStep}]
    Since $A\in\Iso_\Omega(\alpha)$, we know that $(\mu\boxtimes\lebesgue)\bigl(N_{\Omega\boxtimes I(\delta)}[A]\bigr)=\min_{t\in[0,\gamma]}B_\alpha(t)$, so we look to the equality criteria from \Cref{mainBound}.

    We begin by noting that $u_A(0^+)=y$ while $u_A(\gamma^-)=x$.
    Furthermore,
    \[
        \int_0^\gamma u_A(t)\ dt = \xi y+(\gamma-\xi)x,
    \]
    and so $\xi=\xi(u_A)$.
    Additionally,
    \[
        \alpha=\int_0^1 u_A(t)\ dt=\xi y+(1-\xi)x\implies x={\alpha-\xi u_A(0^+)\over 1-\xi}.
    \]
    By assumption, $0<x<y<1$ and so the claim is simply a restatement of the equality criteria from \Cref{mainBound} applied to $A$.
\end{proof}

\section{Cuboids and anti-cuboids}

\begin{defn}[Cuboids and anti-cuboids]
    Fix numbers $x_1,\dots,x_n\in[0,1]$.
    The cuboid with side-lengths $x_1,\dots,x_n$ is defined to be
    \[
        \cube{x_1,\dots,x_n}\eqdef\prod_{i=1}^n[0,x_i],
    \]
    and the anti-cuboid with side-lengths $x_1,\dots,x_n$ is defined to be
    \[
        \anticube{x_1,\dots,x_n}\eqdef\bigcup_{i=1}^n\{a\in[0,1]^n:a_i<x_i\}=[0,1]^n\setminus\prod_{i=1}^n[x_i,1].
    \]
\end{defn}
Observe that the anti-cuboid with side-lengths $x_1,\dots,x_n$ is essentially the complement of the cuboid with side-lengths $1-x_1,\dots,1-x_n$.
\medskip

Already, we should note the following simple fact that will be necessary going forward:
\begin{prop}\label[prop]{bigSeesAll}
   Fix numbers $\delta_1,\dots,\delta_n\in(0,1)$.
   We have $\lebesgue^n\bigl(N_{I(\delta_1,\dots,\delta_n)}[X]\bigr)=1$ if either of the following 
   \begin{enumerate}
       \item $X$ contains a measurable subset $X'$ with $\lebesgue^n(X')\geq 1-\prod_{i=1}^n\delta_i$, or
       \item $X$ is a cuboid with $\lebesgue^n(X)\geq 1-\min_{i\in[n]}\delta_i$.
   \end{enumerate}
\end{prop}
\begin{proof}
   We prove the contrapositive, so suppose that $\lebesgue^n\bigl(N[X]\bigr)<1$.
   In particular, there is some $y\notin N[X]$, which implies that $N[y]\cap X=\varnothing$.
   Naturally, $\lebesgue^n\bigl(N[y]\bigr)\geq \prod_{i=1}^n\delta_i$ and so $X$ is disjoint from a set of measure $\prod_{i=1}^n\delta_i$ which implies the first item.

   Now, suppose that $X=\cube{x_1,\dots,x_n}$ for some $x_1,\dots,x_n\in[0,1]$.
   It is easy to see that
   \[
       \lebesgue^n\bigl(N[X]\bigr)=\prod_{i=1}^n\min\{1,x_i+\delta_i\};
   \]
   thus, since this quantity is strictly less than $1$, there must be some $i\in[n]$ for which $x_i<1-\delta_i$.
   Therefore,
   \[
       \lebesgue^n(X)=\prod_{i=1}^nx_i<1-\delta_i\leq1-\min_{i\in[n]}\delta_i.\qedhere
   \]
\end{proof}

For any numbers $\delta_1,\dots,\delta_n\in(0,1)$, define the functions $\ncube{\delta_1,\dots,\delta_n},\nanticube{\delta_1,\dots,\delta_n}\colon[0,1]^n\to[0,1]$ by
\begin{align*}
    \ncube{\delta_1,\dots,\delta_n}(x_1,\dots,x_n) &\eqdef \lebesgue^n\bigl(N_{I(\delta_1,\dots,\delta_n)}\bigl[\cube{x_1,\dots,x_n}\bigr]\bigr)=\prod_{i=1}^n\min\{1,x_i+\delta_i\},\\
    \nanticube{\delta_1,\dots,\delta_n}(x_1,\dots,x_n) &\eqdef \lebesgue^n\bigl(N_{I(\delta_1,\dots,\delta_n)}\bigl[\anticube{x_1,\dots,x_n}\bigr]\bigr)=1-\prod_{\substack{i\in[n]:\\ x_i>0}}\max\{0,1-x_i-\delta_i\}.
\end{align*}

\begin{defn}[Efficient cuboids and anti-cuboids]
    Fix numbers $0<\delta_1\leq\dots\leq\delta_n<1$.

    The cuboid $\cube{x_1,\dots,x_n}$ is said to be \emph{efficient in $I(\delta_1,\dots,\delta_n)$} if there is some $t\in\{0,\dots,n\}$ such that $x_i<1-\delta_i$ if $i\leq t$ and $x_i=1$ if $i>t$.

    The anti-cuboid $\anticube{x_1,\dots,x_n}$ is said to be \emph{efficient in $I(\delta_1,\dots,\delta_n)$} if there is some $t\in\{0,\dots,n\}$ such that $0<x_i<1-\delta_i$ if $i\leq t$ and $x_i=0$ if $i>t$.
\end{defn}

We first justify this terminology by showing that efficient cuboids and anticuboids can be made at least as good for our problem as any other.

\begin{lemma}\label[lemma]{efficientAreBetter}
    Fix numbers $0<\delta_1\leq\dots\leq\delta_n<1$ and a vector $\vec x\in[0,1]^n$.
    \begin{enumerate}
        \item If $\prod_{i=1}^nx_i<1-\delta_1$, then there is an efficient cuboid $C$ in $I(\delta_1,\dots,\delta_n)$ satisfying
            \[
                \lebesgue^n(C)=\lebesgue^n(\cube{\vec x})\qquad\text{and}\qquad\lebesgue^n\bigl(N_{I(\delta_1,\dots,\delta_n)}[C]\bigr)\leq\lebesgue^n\bigl(N_{I(\delta_1,\dots,\delta_n)}[\cube{\vec x}]\bigr).
            \]
        \item If $1-\prod_{i=1}^n(1-x_i)<1-\prod_{i=1}^n\delta_i$, then there is an efficient anti-cuboid $A$ in $I(\delta_1,\dots,\delta_n)$ satisfying
            \[
                \lebesgue^n(A)=\lebesgue^n(\anticube{\vec x})\qquad\text{and}\qquad\lebesgue^n\bigl(N_{I(\delta_1,\dots,\delta_n)}[A]\bigr)\leq\lebesgue^n\bigl(N_{I(\delta_1,\dots,\delta_n)}[\anticube{\vec x}]\bigr).
            \]
    \end{enumerate}
\end{lemma}
\begin{proof}
    We focus first on cuboids.
    Enumerate the set $\{i\in[n]:x_i<1-\delta_i\}=\{t_1,\dots,t_k\}$ where $t_1<\dots<t_k$; note that we could have $k=0$.
    Define $x_i'=x_{t_i}$ for $i\leq k$ and $x_i'=1$ for $i>k$.
    Certainly $\lebesgue^n(\cube{\vec x'})\geq\lebesgue^n(\cube{\vec x})$.
    Furthermore, since $\delta_1\leq\dots\leq\delta_n$,
    \[
        \lebesgue^n\bigl(N[\cube{\vec x}]\bigr)=\prod_{i=1}^k(x_{t_i}+\delta_{t_i})\geq\prod_{i=1}^k(x_i'+\delta_i)=\lebesgue^n\bigl(N[\cube{\vec x'}]\bigr).
    \]
    Additionally, $\cube{\vec x'}$ is clearly efficient in $I(\delta_1,\dots,\delta_n)$.

    Now, define $x_1''=x_1'\cdot\lebesgue^n(\cube{\vec x})/\lebesgue^n(\cube{\vec x'})$ and $x_i''=x_i'$ for all $i\neq 1$.
    By construction $\lebesgue^n(\cube{\vec x''})=\lebesgue^n(\cube{\vec x})$ and $\lebesgue^n\bigl(N[\cube{\vec x''}]\bigr)\leq\lebesgue^n\bigl(N[\cube{\vec x}]\bigr)$, so we just need to show that $\cube{\vec x''}$ is efficient.
    If $x_1'<1-\delta_1$, then $\cube{\vec x''}$ is clearly efficient since $x_1''\leq x_1'$.
    Otherwise, $x_1'=1$ which implies that $x_i'=1$ for all $i\in[n]$, meaning that $\lebesgue^n(\cube{\vec x'})=1$.
    Therefore, $x_1''=\lebesgue^n(\cube{\vec x})<1-\delta_1$ by assumption and so $\cube{\vec x''}$ is efficient.
    \medskip

    We now consider anti-cuboids.
    Suppose first that there is some $i\in[n]$ for which $x_i\geq 1-\delta_i$.
    Then $N[\anticube{\vec x}]=[0,1]^n$ and so we simply need to show that there is an efficient anti-cuboid $A$ with $\lebesgue^n(A)=\lebesgue^n(\anticube{\vec x})$.
    Set $M=\lebesgue^n(\anticube{\vec x})$ and let $t\in[n]$ be the smallest integer for which $M<1-\prod_{i=1}^t\delta_i$; note that $t$ exists since $M<1-\prod_{i=1}^n\delta_i$ by assumption.
    Naturally, we may find some $\epsilon>0$ for which $M=1-(1+\epsilon)^t\prod_{i=1}^t\delta_i$.
    Define $x_i'=1-(1+\epsilon)\delta_i$ for each $i\leq t$ and $x_i'=0$ for each $i>t$.
    By construction, $\lebesgue^n(\anticube{\vec x'})=M$, so we must show that $\anticube{\vec x'}$ is efficient.
    Since $\epsilon>0$, certainly $x_i'<1-\delta_i$ for each $i\leq t$.
    Now, suppose for the sake of contradiction that there is some $i\leq t$ for which $x_i'\leq 0$, i.e.\ $(1+\epsilon)\delta_i\geq 1$.
    Since $\delta_1\leq\dots\leq\delta_n$, we would, in particular, have $(1+\epsilon)\delta_t\geq 1$.
    In this case, however,
    \[
        M=1-(1+\epsilon)^t\prod_{i=1}^t\delta_i\leq 1-(1+\epsilon)^{t-1}\prod_{i=1}^{t-1}\delta_i<1-\prod_{i=1}^{t-1}\delta_i,
    \]
    which contradicts the definition of $t$.

    With this out of the way, we may suppose that $x_i<1-\delta_i$ for all $i\in[n]$.
    Enumerate the set $\{i\in[n]:x_i>0\}=\{t_1,\dots,t_k\}$ where $t_1<\dots<t_k$; note that we could have $k=0$.
    Define $x_i'=x_{t_i}$ for $i\leq k$ and $x_i'=0$ for $i>k$.
    Certainly $\lebesgue^n(\anticube{\vec x'})=\lebesgue^n(\anticube{\vec x})$ and $\anticube{\vec x'}$ is efficient in $I(\delta_1,\dots,\delta_n)$.
    Finally, since $\delta_1\leq\dots\leq\delta_n$,
    \[
        \lebesgue^n\bigl(N[\anticube{\vec x}]\bigr)=1-\prod_{i=1}^k(1-x_{t_i}-\delta_{t_i})\geq 1-\prod_{i=1}^k(1-x_i'-\delta_i)=\lebesgue^n\bigl(N[\anticube{\vec x'}]\bigr),
    \]
    which establishes the claim.
\end{proof}

Next, for numbers $\delta_1,\dots,\delta_n\in(0,1)$, define the functions $\bcube{\delta_1,\dots,\delta_n},\banticube{\delta_1,\dots,\delta_n}\colon[0,1]\to\R\cup\{+\infty\}$ by
\begin{align*}
    \bcube{\delta_1,\dots,\delta_n}(x) & \eqdef\begin{cases}
        \bigl(x^{1/n}+\bigl(\prod_{i=1}^n\delta_i\bigr)^{1/n}\bigr)^n, & \text{if }0\leq x\leq\prod_{i=1}^n{\delta_i\over\delta_n},\\
        +\infty, & \text{otherwise}.
    \end{cases}\\
    \banticube{\delta_1,\dots,\delta_n}(x) & \eqdef\begin{cases}
        1-\bigl((1-x)^{1/n}-\bigl(\prod_{i=1}^n\delta_i\bigr)^{1/n}\bigr)^n, & \text{if }0\leq x\leq 1-\prod_{i=1}^n\delta_i,\\
        +\infty, & \text{otherwise}.
    \end{cases}
\end{align*}
We extend these definitions to include the case when $n=0$ by defining
\[
    \bcube{()}\eqdef 1\qquad\text{and}\qquad \banticube{()}\eqdef 1.
\]

The functions $\bcube{\delta_1,\dots,\delta_t}$ and $\banticube{\delta_1,\dots,\delta_t}$ compute the size of the neighborhood of particular cuboids and anticuboids.
They thus yield upper bounds on the isoperimetric inequality of $I(\delta_1,\dots,\delta_n)$.
\begin{theorem}\label[theorem]{balancedUpperBound}
    Fix an integer $n\geq 1$ and number $0<\delta_1\leq\dots\leq\delta_n<1$.
    \[
        \iso_{I(\delta_1,\dots,\delta_n)}\leq\min_{t\in\{0,\dots,n\}}\min\bigl\{\bcube{\delta_1,\dots,\delta_t},\ \banticube{\delta_1,\dots,\delta_t}\bigr\}.
    \]
\end{theorem}
\begin{proof}
    Certainly $\iso_{I(\delta_1,\dots,\delta_n)}\leq 1=\bcube{()}=\banticube{()}$ always, so fix $t\in[n]$ and $\alpha\in[0,1]$.

    If $\alpha>\prod_{i=1}^t{\delta_i\over\delta_t}$, then certainly $\iso_{I(\delta_1,\dots,\delta_n)}(\alpha)\leq +\infty=\bcube{\delta_1,\dots,\delta_t}(\alpha)$.
    Otherwise, $\alpha\leq\prod_{i=1}^t{\delta_i\over\delta_t}$; define $x_i=\delta_i\bigl({\alpha\over\prod_{i=1}^t\delta_i}\bigr)^{1/t}$ for $i\leq t$ and $x_i=1$ for $i>t$.
    Note that for each $i\leq t$, $x_i\leq\delta_i (1/\delta_t)\leq 1$ and so the cuboid $\cube{\vec x}$ exists and has $\lebesgue^n(\cube{\vec x})=\alpha$.
    Therefore,
    \begin{align*}
        \iso_{I(\delta_1,\dots,\delta_n)}(\alpha) &\leq \lebesgue^n\bigl(N[\cube{\vec x}]\bigr)\leq \prod_{i=1}^t(x_i+\delta_i)=\prod_{i=1}^t\biggl(\delta_i\biggl({\alpha\over \prod_{j=1}^t\delta_j}\biggr)^{1/t}+\delta_i\biggr)\\
                                                  &= \prod_{i=1}^t\delta_i\cdot\prod_{i=1}^t\biggl(\biggl({\alpha\over\prod_{j=1}^t\delta_j}\biggr)^{1/t}+1\biggr) = \biggl(\alpha^{1/t}+\biggl(\prod_{i=1}^t\delta_i\biggr)^{1/t}\biggr)^t\\
                                                  &=\bcube{\delta_1,\dots,\delta_t}(\alpha).
    \end{align*}

    Similarly, if $\alpha>1-\prod_{i=1}^t\delta_i$, then certainly $\iso_{I(\delta_1,\dots,\delta_n)}(\alpha)\leq +\infty=\banticube{\delta_1,\dots,\delta_t}(\alpha)$.
    Otherwise, $\alpha\leq 1-\prod_{i=1}^t\delta_i$; define $x_i=1-\delta_i\bigl({1-\alpha\over\prod_{i=1}^t\delta_i}\bigr)^{1/t}$ for $i\leq t$ and $x_i=0$ for $i>t$.
    Note that for each $i\leq t$, $1\geq 1-x_i=\delta_i\bigl({1-\alpha\over\prod_{i=1}^t\delta_i}\bigr)^{1/t}\geq\delta_i$.
    In particular, the anticuboid $\anticube{\vec x}$ exists and has $\lebesgue^n(\anticube{\vec x})=\alpha$.
    Therefore, since $1-x_i\geq\delta_i$,
    \begin{align*}
        \iso_{I(\delta_1,\dots,\delta_n)}(\alpha) &\leq \lebesgue^n\bigl(N[\anticube{\vec x}]\bigr)= 1-\prod_{i=1}^t(1-x_i-\delta_i)=1-\prod_{i=1}^t\biggl(\delta_i\biggl({1-\alpha\over\prod_{j=1}^t\delta_j}\biggr)^{1/t}-\delta_i\biggr)\\
                                                  &= 1-\prod_{i=1}^t\delta_i\cdot\prod_{i=1}^t\biggl(\biggl({1-\alpha\over\prod_{j=1}^t\delta_j}\biggr)^{1/t}-1\biggr)=1-\biggl((1-\alpha)^{1/t}-\biggl(\prod_{i=1}^t\delta_i\biggr)^{1/t}\biggr)\\
                                                  &=\banticube{\delta_1,\dots,\delta_t}(\alpha).\qedhere
    \end{align*}
\end{proof}

We show next that these particular cuboids/anticuboids are optimal among all cuboids/anticuboids.

\begin{theorem}\label[theorem]{balancedCuboidsAreBest}
    Fix an integer $n\geq 1$ and numbers $0<\delta_1\leq\dots\leq\delta_n<1$.
    For any $\alpha\in(0,1]$ and any $x_1,\dots,x_n\in[0,1]$ with $\prod_{i=1}^n x_i=\alpha$,
    \[
        \ncube{\delta_1,\dots,\delta_n}(x_1,\dots,x_n) \geq\min_{t\in\{0,\dots,n\}}\bcube{\delta_1,\dots,\delta_t}(\alpha).
    \]
\end{theorem}
\begin{proof}
    Consider the optimization problem which asks to minimize $\ncube{\delta_1,\dots,\delta_n}(x_1,\dots,x_n)$ subject to $x_1,\dots,x_n\in[0,1]$ and $\prod_{i=1}^n x_i=\alpha$.
    Of course, $\ncube{\delta_1,\dots,\delta_n}$ is a continuous function and the set $\{\vec x\in[0,1]^n:\prod_{i=1}^n x_i=\alpha\}$ is compact and so there exists an optimizer $\vec x^*$.
    We can assume $\prod_{i=1}^n x_i < 1-\delta_1$, in order to apply  \Cref{efficientAreBetter},
    for otherwise part 2 of
    \Cref{bigSeesAll}
    implies $\ncube{\delta_1,\dots,\delta_n}(x_1,\dots,x_n)=1$ and the conclusion is immediate.
    So \Cref{efficientAreBetter} means we can select $\vec x^*$ so that the cuboid $\cube{\vec x^*}$ is efficient; let $t\in[n]$ be such that $0<x_i^*<1-\delta_i$ for all $i\leq t$ and $x_i^*=1$ for all $i>t$.
    Define the function $f\colon[0,1]^t\to\R$ by
    \[
        f(x_1,\dots,x_t)=\prod_{i=1}^t(x_i+\delta_i),
    \]
    which is clearly continuously differentiable.
    Furthermore, since $0<x_i^*<1-\delta_i$ for all $i\in[t]$, we find that $f(x_1,\dots,x_t)=\ncube{\delta_1,\dots,\delta_n}(x_1,\dots,x_t,1,\dots,1)$ for all $(x_1,\dots,x_t)$ in an open neighborhood about $(x_1^*,\dots,x_t^*)$.
    In particular, $(x_1^*,\dots,x_t^*)$ must be a \emph{local} optimizer of the problem $\min f(x_1,\dots,x_t)$ subject to $\prod_{i=1}^t x_i=\alpha$.
    Since $f$ is continuously differentiable, we may apply the method of Lagrange multipliers to find that there is some $\lambda\in\R$ for which
    \[
        \prod_{j\neq i}(x_j^*+\delta_j)=\lambda\prod_{j\neq i}x_j^*\qquad\text{for all }i\in[t].
    \]
Multiplying both sides of this equality by $x_i^*(x_i^*+\delta_i)$ then implies that
    \[
        f(x_1^*,\dots,x_t^*)\cdot x_i^*=\lambda\alpha(x_i^*+\delta_i)\qquad\text{for all }i\in[t].
    \]
    In turn, this implies that there is some fixed number $x$ for which $x_i^*=\delta_i x$ for all $i\in[t]$.
    Naturally, $\prod_{i=1}^t x_i^*=\alpha\implies x=\bigl(\alpha/\prod_{i=1}^t\delta_i\bigr)^{1/t}$.

    Putting everything together, we have shown that
    \begin{align*}
        \ncube{\delta_1,\dots,\delta_n}(x_1,\dots,x_n) &\geq \ncube{\delta_1,\dots,\delta_n}(x_1^*,\dots,x_n^*)=f(x_1^*,\dots,x_t^*)\\
                                                       &= \prod_{i=1}^t\biggl(\delta_i\biggl({\alpha\over \prod_{j=1}^t\delta_j}\biggr)^{1/t}+\delta_i\biggr)=\prod_{i=1}^t\delta_i\cdot\prod_{i=1}^t\biggl(\biggl({\alpha\over\prod_{j=1}^t\delta_j}\biggr)^{1/t}+1\biggr)\\
                                                       &= \biggl(\alpha^{1/t}+\biggl(\prod_{i=1}^t\delta_i\biggr)^{1/t}\biggr)^t=\bcube{\delta_1,\dots,\delta_t}(\alpha).\qedhere
    \end{align*}
\end{proof}

In the same vein, we show the following.
\begin{theorem}\label[theorem]{balancedAnticuboidsAreBest}
    Fix an integer $n\geq 1$ and numbers $0<\delta_1\leq\dots\leq\delta_n<1$.
    For any $\alpha\in(0,1]$ and any $x_1,\dots,x_n\in[0,1]$ with $1-\prod_{i=1}^n(1-x_i)=\alpha$,
    \[
        \nanticube{\delta_1,\dots,\delta_n}(x_1,\dots,x_n) \geq\min_{t\in\{0,\dots,n\}}\banticube{\delta_1,\dots,\delta_t}(\alpha).
    \]
\end{theorem}
\begin{proof}
    Consider the optimization problem which asks to minimize $\nanticube{\delta_1,\dots,\delta_n}$ over the set $X\eqdef\{\vec x\in[0,1]^n:1-\prod_{i=1}^n(1-x_i)=\alpha\}$.
    While $\nanticube{\delta_1, \dots, \delta_n}$ is not continuous on $[0,1]^n$,
    we nevertheless begin by demonstrating that a minimizer to this problem does indeed exist.

    For $S\subseteq[n]$ define the function and set
    \[
        f_S(x_1,\dots,x_n)\eqdef 1-\prod_{i\in S}\max\{0,1-x_i-\delta_i\},\quad X_S\eqdef\biggl\{\vec x\in[0,1]^n:1-\prod_{i=1}^n(1-x_i)=\alpha\text{ and }x_i=0\text{ if }i\notin S\biggr\}.
    \]

    Firstly, for any $S\subseteq[n]$ and any $\vec x\in X_S$, it is easy to see that $\vec x\in X$ and that $f_S(\vec x)\leq\nanticube{\delta_1,\dots,\delta_n}(\vec x)$.
    On the other hand, if $\vec x\in X$, then with $S=\{i\in[n]:x_i>0\}$, we find that $\vec x\in X_S$ and that $\nanticube{\delta_1,\dots,\delta_n}(\vec x)=f_S(\vec x)$.

    Thus, since each $f_S$ is continuous and each $X_S$ is closed, we find that there is some $\vec x^*$ which minimizes $\nanticube{\delta_1,\dots,\delta_n}(\vec x)$ subject to $\vec x\in X$.
    We may furthermore apply \Cref{efficientAreBetter} to see that we may select $\vec x^*$ so that the anti-cuboid $\anticube{\vec x^*}$ is efficient.
    From here, the proof is identical to the proof of \Cref{balancedCuboidsAreBest}.
\end{proof}

Define the function $\qiso{\delta_1,\dots,\delta_n}\colon[0,1]\to[0,1]$ by
\[
    \qiso{\delta_1,\dots,\delta_n}(x) \eqdef \begin{cases}
        0, & \text{if }x=0,\\
        \min_{t\in\{0,\dots,n\}}\min\bigl\{\bcube{\delta_1,\dots,\delta_t}(x),\ \banticube{\delta_1,\dots,\delta_t}(x)\bigr\}, & \text{otherwise}.
    \end{cases}
\]

That is to say, $\iso{\delta_1, \dots, \delta_n}$ computes the smalleest possible neighborhood among
We seek to prove that $\iso_{I(\delta_1,\dots,\delta_n)}=\qiso{\delta_1,\dots,\delta_n}$ provided that $\delta_1\leq\dots\leq\delta_n$, which we will do in the next section.
The main piece of the argument here hinges on \Cref{isotightClassification,optimalStep}, so we must verify that $\qiso{\delta_1,\dots,\delta_n}$ satisfies the hypotheses of those theorems.

\begin{theorem}\label[theorem]{qisoIsNice}
    Fix an integer $n\geq 1$ and numbers $0<\delta_1\leq\dots\leq\delta_n<1$.
    The following hold:
    \begin{enumerate}
        \item $\qiso{\delta_1,\dots,\delta_n}$ is continuous on $(0,1]$ and concave on $[0,1]$.
        \item Fix a non-degenerate interval $I\subseteq[0,1]$.
            If $\qiso{\delta_1,\dots,\delta_n}$ is affine on $I$, then either
            \begin{itemize}
                \item $I\subseteq\bigl[1-\prod_{i=1}^n\delta_i,\ 1\bigr]$ and $\qiso{\delta_1,\dots,\delta_n}(x)=1$ for all $x\in I$, or
                \item $I\subseteq\bigl(0,1-\delta_1\bigr]$ and $\qiso{\delta_1,\dots,\delta_n}(x)=x+\delta_1$ for all $x\in I$.
            \end{itemize}
    \end{enumerate}
\end{theorem}

In order to prove this, we will need to generalize two well-known facts: the minimum of continuous functions is continuous, and the minimum of concave functions is concave.
\begin{prop}\label[prop]{minWithDifferentDomains}
    Let $I_1,\dots,I_n\subseteq\R$ be non-degenerate, closed intervals with $I_1\subseteq\dots\subseteq I_n$ and suppose that $f_i\colon I_i\to\R$ are functions.
    For each $i\in[n]$, define the function $\breve f_i\colon I_n\to\R\cup\{+\infty\}$ by
    \[
        \breve f_i(x)=\begin{cases}
            f_i(x), & \text{if }x\in I_i,\\
            +\infty, & \text{otherwise}.
        \end{cases}
    \]
    Suppose the following property holds for every $i\in[n]$:
    \begin{equation}\label{cond:smaller}
        x_i\in \partial I_i\setminus\partial I_n\qquad\implies\qquad
                f_i(x_i)>\min_{j\in[n]}\breve f_j(x_i)
        \quad\text{or}\quad f_i(x_i)\geq\min_{j\in\{i+1,\dots,n\}}f_j(x_i).
    \end{equation}
    Then the following hold:
    \begin{enumerate}
        \item Fix any $x\in I_n$. If $f_i$ is continuous at $x$ for each $i\in[n]$ satisfying $x\in I_i$, then $\min_{i\in[n]}\breve f_i$ is also continuous at $x$.
        \item If $f_i$ is concave on $I_i$ for each $i\in[n]$, then $\min_{i\in[n]}\breve f_i$ is concave on $I_n$.
    \end{enumerate}
\end{prop}
The proof is not too difficult, but is somewhat lengthy.
We therefore relegate it to \Cref{sec:minWithDifferentDomains} and instead move on to proving \Cref{qisoIsNice}.
Naturally, we will apply \Cref{minWithDifferentDomains} to the functions $\bcube{\vec\delta}$ and $\banticube{\vec\delta}$, and so we begin by verifying part of the hypothesis.

\begin{lemma}\label[lemma]{hypothesisCheck}
    Fix an integer $n\geq 1$ and numbers $\delta_1\leq\dots\leq\delta_n\in(0,1)$.
    With $x=\prod_{i=1}^n{\delta_i\over\delta_n}$,
    \[
        \bcube{\delta_1,\dots,\delta_n}(x)>\min_{t\in\{0,\dots,n\}}\bcube{\delta_1,\dots,\delta_t}(x).
    \]
\end{lemma}
\begin{proof}
    For each $i\in[n]$, set $x_i=\delta_i/\delta_n$, so $0\leq x_i\leq 1$ for each $i\in[n]$ since $\delta_1\leq\dots\leq\delta_n$.
    Note that $\prod_{i=1}^n x_i=x$ by construction.
    Also $x_i=\delta_i/\delta_n=\delta_i\bigl(x/\prod_{j=1}^n\delta_j\bigr)^{1/n}$, and so, since $x_n+\delta_n=1+\delta_n>1$, we bound
    \begin{align*}
        \ncube{\delta_1,\dots,\delta_n}(x_1,\dots,x_n) &= \prod_{i=1}^n\min\bigl\{1,x_i+\delta_i\bigr\}<\prod_{i=1}^n(x_i+\delta_i)=\prod_{i=1}^n\biggl(\delta_i\biggl({x\over\prod_{j=1}^n\delta_j}\biggr)^{1/n}+\delta_i\biggr)\\
                                                       &= \biggl(\prod_{i=1}^n\delta_i\biggr)\cdot\biggl(\biggl({x\over\prod_{j=1}^n\delta_j}\biggr)^{1/n}+1\biggr)^n=\biggl(x^{1/n}+\biggl(\prod_{i=1}^n\delta_i\biggr)^{1/n}\biggr)^n\\
                                                       &=\bcube{\delta_1,\dots,\delta_n}(x).
    \end{align*}
    Combining this inequality with \Cref{balancedCuboidsAreBest} then yields
    \[
        \bcube{\delta_1,\dots,\delta_n}(x)>\ncube{\delta_1,\dots,\delta_n}(x_1,\dots,x_n)\geq\min_{t\in\{0,\dots,n\}}\bcube{\delta_1,\dots,\delta_n}(x).\qedhere
    \]
\end{proof}

We can now prove \Cref{qisoIsNice}.
\begin{proof}[Proof of \Cref{qisoIsNice}]
    We begin by defining some functions and intervals.
    For $t\in[n]$, define
    \[
        I_t=\biggl[0,\prod_{i=1}^t{\delta_i\over\delta_t}\biggr],\qquad c_t(x)=\biggl(x^{1/t}+\biggl(\prod_{i=1}^t\delta_i\biggr)^{1/t}\biggr)^t
    \]
    \[
        J_t=\biggl[0,1-\prod_{i=1}^t\delta_i\biggr],\qquad a_t(x)=1-\biggl((1-x)^{1/t}-\biggl(\prod_{i=1}^t\delta_i\biggr)^{1/t}\biggr)^t.
    \]
    Additionally define $I_0=J_0=[0,1]$ and $c_0(x)=a_0(x)=1$.
    Note that $c_t\colon I_t\to\R$ and $a_t\colon J_t\to\R$ are continuous and concave for all $t\in\{0,\dots,n\}$ (the concavity can be seen by checking the second derivative).
    Additionally, using the notation from \Cref{minWithDifferentDomains}, we see that $\breve c_t=\bcube{\delta_1,\dots,\delta_t}$ and $\breve a_t=\banticube{\delta_1,\dots,\delta_t}$ for each $t\in\{0,\dots,n\}$.
    Finally, note that the $I_t$'s are nested, as are the $J_t$'s.

    Define the functions $c,a\colon[0,1]\to\R$ by $c=\min_{t\in\{0,\dots,n\}}\breve c_t$ and $a=\min_{t\in\{0,\dots,n\}}\breve a_t$.
    \Cref{hypothesisCheck} shows that the hypothesis of \Cref{minWithDifferentDomains} is satisfied for the functions $c_0,\dots,c_n$ and so $c$ is continuous and concave on $[0,1]$.
    The hypotheses of \Cref{minWithDifferentDomains} are also satisfied for the functions $a_0,\dots,a_n$ since if $x=1-\prod_{i=1}^t\delta_i$, then $a_t(x)=1=a_0(x)$; thus $a$ is continuous and concave on $[0,1]$ as well.

    Since $\qiso{\delta_1,\dots,\delta_n}(0)=0$ and $\qiso{\delta_1,\dots,\delta_n}(x)=\min\{c(x),\ a(x)\}$ otherwise, this verifies that $\qiso{\delta_1,\dots,\delta_n}$ is continuous on $(0,1]$ and concave on $[0,1]$.
    \medskip

    Now, fix a non-degenerate interval $I\subseteq[0,1]$ and suppose that $\qiso{\delta_1,\dots,\delta_n}$ is affine on $I$.
    Note that for each $t\geq 2$, the functions $c_t$ and $a_t$ are strictly concave on $I_t$ and $J_t$, respectively (this can be seen by checking the second derivative).
    In particular, if $\qiso{\delta_1,\dots,\delta_n}$ is affine on $I$, then it must be the case that $\qiso{\delta_1,\dots,\delta_n}\in\{c_0,c_1,a_0,a_1\}$ when restricted to $I$.
    Since $c_0(x)=a_0(x)=1$ and $c_1(x)=a_1(x)=x+\delta_1$, the claim follows.
\end{proof}

We will rely on the continuity and concavity of $\qiso{\delta_1,\dots,\delta_n}$ directly in the next section.
The observations about when $\qiso{\delta_1,\dots,\delta_n}$ is affine will be used only to derive the following fact related to \Cref{optimalStep}:
\begin{corollary}\label[corollary]{notAffine}
    Fix numbers $0<\xi<\gamma<1$, $\alpha\in[0,1]$ and $0<\delta_1\leq\dots\leq\delta_n<1$.
    Set $Q=\qiso{\delta_1,\dots,\delta_n}$ and fix a non-degenerate interval $I$.
    If
    \[
        (\gamma-\xi)Q(t)+(1-\gamma+\xi)Q\biggl({\alpha-(1-\xi)t\over\xi}\biggr)=Q(\alpha)
    \]
    for all $t\in I$, then $Q(t)=1$ for all $t\in I$.
\end{corollary}
\begin{proof}
    Since $Q$ is concave on $[0,1]$ (\Cref{qisoIsNice}), we know that both $Q(t)$ and $Q\bigl({\alpha-(1-\xi)t\over\xi}\bigr)$ must be concave on $I$.
    Since $\gamma-\xi>0$ and $1-\gamma+\xi>0$ by assumption, this means that both $Q(t)$ and $Q\bigl({\alpha-(1-\xi)t\over\xi}\bigr)$ must, in fact, be affine on $I$.
    Thus, due to \Cref{qisoIsNice}, we know that either $Q(t)=1$ for all $t\in I$ or $Q(t)=t+\delta_1$ for all $t\in I$.
    Suppose for the sake of contradiction that $Q(t)=t+\delta_1$, in which case $Q\bigl({\alpha-(1-\xi)t\over\xi}\bigr)={\alpha-(1-\xi)t\over\xi}+\delta_1$.
    Therefore,
    \begin{align*}
        Q(\alpha) &=  (\gamma-\xi)\bigl(t+\delta_1\bigr)+(1-\gamma+\xi)\biggl({\alpha-(1-\xi)t\over\xi}+\delta_1\biggr)\\
                  &= {\gamma-1\over\xi}t+\biggl(\alpha+\delta_1+{\alpha(1-\gamma)\over\xi}\biggr),
    \end{align*}
    for all $t\in I$.
    However, since $I$ is non-degenerate, the linear term must vanish and so ${\gamma-1\over\xi}=0\implies\gamma=1$; a contradiction.
\end{proof}

\section{The isoperimetric inequality}

This section is dedicated to the proof of the following theorem.
\begin{theorem}\label[theorem]{qualitativeIsoperimetric}
    Fix an integer $n\geq 1$ and numbers $\delta_1,\dots,\delta_n\in(0,1)$.
    \begin{enumerate}
        \item $I(\delta_1,\dots,\delta_n)$ is iso-tight, i.e.\ $\Iso_{I(\delta_1, \dots, \delta_n)}(\alpha) \neq \varnothing$ for every $\alpha \in (0,1]$, and
        \item For every $\alpha\in(0,1]$, there is some element of $\Iso_{I(\delta_1,\dots,\delta_n)}(\alpha)$ which is either a cuboid or an anti-cuboid.
    \end{enumerate}
\end{theorem}
Without loss of generality, we may suppose that $\delta_1\leq\dots\leq\delta_n$.

Observe that if the above theorem holds for some fixed $n$, then $\iso_{I(\delta_1,\dots,\delta_n)}=\qiso{\delta_1,\dots,\delta_n}$.
Indeed, \Cref{balancedUpperBound} tells us that $\iso_{I(\delta_1,\dots,\delta_n)}\leq\qiso{\delta_1,\dots,\delta_n}$ and the lower bound follows from \Cref{balancedCuboidsAreBest,balancedAnticuboidsAreBest}.
We will need to use this fact in our inductive step.
\medskip

As mentioned the proof is by induction on $n$ with the base-case of $n=1$ being the content of \Cref{1d}.
Therefore, suppose that $n\geq 2$ and fix numbers $0<\delta_1\leq\dots\leq\delta_n<1$.
\medskip

Define $\Omega\eqdef I(\delta_1,\dots,\delta_{n-1})$, so $I(\delta_1,\dots,\delta_n)=\Omega\boxtimes I(\delta_n)$.
By the induction hypothesis, we know that $\Omega$ is iso-tight and that $\iso_\Omega=\qiso{\delta_1,\dots,\delta_{n-1}}$.
\Cref{qisoIsNice} then tells us that $\iso_\Omega$ is continuous on $(0,1]$ and concave on $[0,1]$ and so \Cref{isotightClassification} allows us to conclude that $\Omega\boxtimes I(\delta_n)=I(\delta_1,\dots,\delta_n)$ is iso-tight.
\medskip

Next, fix some $\alpha\in(0,1]$; we must show that $\Iso_{I(\delta_1,\dots,\delta_n)}(\alpha)$ contains either a cuboid or an anti-cuboid.

In addition to telling us that $I(\delta_1,\dots,\delta_n)$ is iso-tight, \Cref{isotightClassification} tells us that there is some $\xi^*\in[0,1-\delta_n]$ so that the following holds:
\begin{itemize}
    \item If $\xi^*\geq\alpha$, then for any $S\in\Iso_\Omega\bigl({\alpha\over \xi^*}\bigr)$, we have $S\times[0,\xi^*]\in\Iso_{I(\delta_1,\dots,\delta_n)}(\alpha)$.
    \item If $\xi^*\leq\alpha$, then for any $S\in\Iso_\Omega\bigl({\alpha-\xi^*\over 1-\xi^*}\bigr)$, we have $\bigl([0,1]^{n-1}\times[0,\xi^*)\bigr)\sqcup(S\times[\xi^*,1])\in\Iso_{I(\delta_1,\dots,\delta_n)}(\alpha)$.
\end{itemize}
The induction hypothesis tells us that, in either case, we can select the $S$ to be either a cuboid or an anti-cuboid.
From here we will break into cases.

\paragraph{Case $\xi^*\geq\alpha$.}
If we can select $S\in\Iso_\Omega\bigl({\alpha\over\xi^*}\bigr)$ so that $S$ is a cuboid, then certainly $X\eqdef S\times[0,\xi^*]$ is also a cuboid, as needed.
Thus, suppose that it is not possible to select $S$ to be a cuboid; we will derive a contradiction.
Thus, suppose that $S=\anticube{x_1,\dots,x_{n-1}}$ for some $x_1,\dots,x_{n-1}\in[0,1]$.
According to \Cref{efficientAreBetter}, we may suppose that $S$ is efficient in $I(\delta_1,\dots,\delta_{n-1})$, so suppose that $t\in\{0,\dots,n-1\}$ is such that $0<x_i<1-\delta_i$ for all $i\leq t$ and $x_i=0$ for all $i>t$.
Certainly, $t\geq 1$ or else $S$ is trivially a cuboid.
Additionally, we may suppose that $t\geq 2$.
Indeed, if $t=1$, then $S=[0,x_1)\times[0,1]^{n-2}$ and so $X=S\times[0,\xi^*]=[0,x_1)\times[0,1]^{n-2}\times[0,\xi^*]$.
As such, the closure of $X$ is a cuboid which has the same measure and neighborhood as $X$, and so $\Iso_{I(\delta_1,\dots,\delta_n)}(\alpha)$ contains this cuboid.

To summarize, we have argued that we can select $S=\anticube{x_1,\dots,x_{n-1}}$ where $0<x_1<1-\delta_1$ and $0<x_2<1-\delta_2$.

Up until now, we have been looking at $I(\delta_1,\dots,\delta_n)$ along its last coordinate; we will now look along its first.
However, in order to use notation set up earlier, we will equivalently consider looking along the last coordinate of $I(\delta_n,\dots,\delta_1)$.

Set $\Omega'\eqdef I(\delta_n,\dots,\delta_2)$, so $I(\delta_n,\dots,\delta_1)=\Omega'\boxtimes I(\delta_1)$.
By the induction hypothesis, $\iso_{\Omega'}=\qiso{\delta_2,\dots,\delta_n}$.
Additionally, the work above establishes that $A\eqdef[0,\xi^*]\times\anticube{x_{n-1},\dots,x_1}$ is an element of $\Iso_{\Omega'\boxtimes I(\delta_1)}(\alpha)$.
Certainly $A\subseteq\Omega'\times[0,1]$ is $\rootedIntervals$-layered and specifically
\[
    \slicex{A}{t}=\begin{cases}
        [0,\xi^*]\times[0,1]^{n-2} & \text{if }t<x_1,\\
        [0,\xi^*]\times\anticube{x_{n-1},\dots,x_2} & \text{if }t>x_1.
    \end{cases}
\]
Setting $a=\lebesgue^{n-2}(\anticube{x_{n-1},\dots,x_2})$, we have
\[
    u_A(t)=\begin{cases}
        \xi^* & \text{if }t<x_1,\\
        \xi^*a & \text{if }t>x_1.
    \end{cases}
\]
Since $0<x_1<1-\delta_1$ and $0<a<1$ (since $0<x_2<1-\delta_2$), we may therefore apply \Cref{optimalStep} (with $x \gets \xi^* a$ and $\xi \gets x_1$) to conclude that
\[
    (1-\delta_1-x_1)\iso_{\Omega'}(t)+(\delta_1+x_1)\iso_{\Omega'}\biggl({\alpha-(1-x_1)t\over x_1}\biggr)=\iso_{\Omega'}(\alpha),
\]
for all $\xi^*a\leq t\leq\alpha$.
However, since $\iso_{\Omega'}=\qiso{\delta_2,\dots,\delta_n}$ by inductive hypothesis, \Cref{notAffine} then implies that $\iso_{\Omega'}(t)=1$ for all $\xi^*a\leq t\leq\alpha$.
Naturally, this means that $\iso_{\Omega'}(t)=1$ for all $\xi^*a\leq t\leq 1$ since $\iso_{\Omega'}$ is increasing.
Therefore,
\[
    \iso_\Omega(\alpha)=\lebesgue^n\bigl(N_\Omega[A]\bigr)\geq(\delta+x_1)\iso_{\Omega'}(\xi^*)+(1-\delta-x_1)\iso_{\Omega'}(\xi^*a)=1,
\]
and so certainly $\Iso_\Omega(\alpha)$ contains both a cuboid and an anti-cuboid.

\paragraph{Case $\xi^*<\alpha$.}
Observe here that
\[
    X\eqdef\bigl([0,1]^{n-1}\times[0,\xi^*)\bigl)\sqcup\bigl(S\times[\xi^*,1]\bigr)=\bigl(S\times[0,1]\bigr)\cup\bigl\{a\in[0,1]^n:a_n<\xi^*\bigr\}.
\]
We conclude that if we can select $S\in\Iso_\Omega\bigl({\alpha-\xi^*\over 1-\xi^*}\bigr)$ to be an anti-cuboid, then $X$ is also an anti-cuboid, as needed.
Thus, suppose that it is not possible to select $S$ to be an anti-cuboid; we will again derive a contradiction.
Of course, we must be able to select $S$ to be a cuboid.
Suppose that $S=\cube{x_1,\dots,x_{n-1}}$ for some $x_1,\dots,x_{n-1}\in[0,1]$.
According to \Cref{efficientAreBetter}, we may suppose that $S$ is efficient in $I(\delta_1,\dots,\delta_{n-1})$, so suppose that $t\in\{0,\dots,n-1\}$ is such that $x_i<1-\delta_i$ for all $i\leq t$ and $x_i=1$ for all $i>t$.
Certainly $t\geq 1$ or else $S$ is trivially an anti-cuboid.
Additionally, we may suppose that $t\geq 2$.
Indeed, if $t=1$, then $S=[0,x_1]\times[0,1]^{n-2}$ and so $X=[0,x_1]\times[0,1]^{n-1}$.
Since $x_1>0$, the set $[0,x_1)\times[0,1]^{n-1}$ is an anti-cuboid which has the same measure and neighborhood as $X$, and so $\Iso_{I(\delta_1,\dots,\delta_n)}(\alpha)$ contains this anti-cuboid.

To summarize, we have argued that we can select $S=\cube{x_1,\dots,x_{n-1}}$ where $0<x_1<1-\delta_1$ and $0<x_2<1-\delta_2$.

Similarly to the case above, we will now look along the last coordinate of $I(\delta_n,\dots,\delta_1)$.
Set $\Omega'\eqdef I(\delta_n,\dots,\delta_2)$, so $I(\delta_n\dots,\delta_1)=\Omega'\boxtimes I(\delta_1)$.
By the induction hypothesis, $\iso_{\Omega'}=\qiso{\delta_2,\dots,\delta_n}$.
Additionally, the work above establishes that $A\eqdef \bigl([0,\xi^*)\times[0,1]^{n-1}\bigr)\sqcup\bigl([\xi^*,1]\times\cube{x_{n-1},\dots,x_1}\bigr)$ is an element of $\Iso_{\Omega'\boxtimes I(\delta_1)}(\alpha)$.
Observe that $A\subseteq\Omega'\times[0,1]$ is $\rootedIntervals$-layered and
\[
    \slicex{A}{t}=\begin{cases}
        \bigl([0,\xi^*)\times[0,1]^{n-2}\bigl)\sqcup\bigl([\xi^*,1]\times\cube{x_{n-1},\dots,x_2}\bigr), & \text{if }t<x_1,\\
        [0,\xi^*)\times[0,1]^{n-2}, & \text{if }t>x_1.
    \end{cases}
\]
Setting $a=\lebesgue^{n-2}(\cube{x_{n-1},\dots,x_2})$, we have
\[
    u_A(t)=\begin{cases}
        \xi^*+(1-\xi^*)a & \text{if }t<x_1,\\
        \xi^* & \text{if }t>x_1.
    \end{cases}
\]
Since $0<x_1<1-\delta_1$ and $0<a<1$ (since $0<x_2<1-\delta_2$), we may therefore apply \Cref{optimalStep} (with $x \gets \xi^*$ and $\xi\gets x_1$) to conclude that
\[
    (1-\delta_1-x_1)\iso_{\Omega'}(t)+(\delta_1+x_1)\iso_{\Omega'}\biggl({\alpha-(1-x_1)t\over x_1}\biggr)=\iso_{\Omega'}(\alpha),
\]
for all $\xi^*\leq t\leq\alpha$.
However, since $\iso_{\Omega'}=\qiso{\delta_2,\dots,\delta_n}$, \Cref{notAffine} then implies that $\iso_{\Omega'}(t)=1$ for all $\xi^*\leq t\leq\alpha$.
Naturally, this means that $\iso_{\Omega'}(t)=1$ for all $\xi^*\leq t\leq 1$ since $\iso_{\Omega'}$ is increasing.
Therefore,
\[
    \iso_\Omega(\alpha)=\lebesgue^n\bigl(N_\Omega[A]\bigr)\geq(\delta_1+x_1)\iso_{\Omega'}(\xi^*+(1-\xi^*)a)+(1-\delta_1-x_1)\iso_{\Omega'}(\xi^*)=1,
\]
and so certainly $\Iso_\Omega(\alpha)$ contains both a cuboid and an anti-cuboid.

\section{Corollaries}\label{corollaries}
In this section, we derive a few corollaries of our main isoperimetric inequality.
Firstly, we can combine \Cref{qualitativeIsoperimetric} and \Cref{circle=interval} to deduce:

\begin{theorem}\label[theorem]{isoperimetricBound}
    Fix numbers $0<\delta_1,\dots,\delta_n\leq 1$.
    If $X_i\in\{I(\delta_i),S(\delta_i/2)\}$ for each $i\in[n]$, then there is either a cuboid or anti-cuboid contained within $\Iso_{X_1\boxtimes\dots\boxtimes X_n}(\alpha)$ for all $\alpha$.
    In particular, if $\delta_1\leq\dots\leq\delta_n$, then
    \[
        \iso_{X_1\boxtimes\dots\boxtimes X_n}(\alpha)=\min_{t\in\{0,\dots,n\}}\min\bigl\{\bcube{\delta_1,\dots,\delta_t}(\alpha),\ \banticube{\delta_1,\dots,\delta_t}(\alpha)\bigr\}
    \]
    for all $\alpha\in(0,1]$.
\end{theorem}
Observe that \Cref{mainCube} is the special case of \Cref{isoperimetricBound} where $X_i=I(\delta)$ for every $i$, and \Cref{mainTorus} is the special case where $X_i=S(\delta)$ for every $i$.
\medskip

We next turn our attention to graph isoperimetric inequalities.
Let $G$ be a graph.
For a subset $A\subseteq V(G)$, we use $N_G[A]$ to denote the closed neighborhood of $A$, that is all vertices in $A$ or adjacent to some vertex in $A$, despite the slight conflict in notation with $N_\Omega$ for a metric space $\Omega$.
For a graph $G$ and a number $\alpha\in[\abs{V(G)}]$, write $\iso_G(\alpha)\eqdef\min\{\abs{N_G[A]}:A\subseteq V(G),\ \abs A\geq\alpha\}$.
For a positive integer $d$, the graph $G^d$ has the same vertex-set as does $G$ and $uv\in E(G^d)$ if the distance between $u$ and $v$ in $G$ is at most $d$.

\begin{theorem}\label{thm:discretize}
    Fix positive integers $k_1, \dots, k_n$ and $d_1,\dots,d_n$ with $d_i\leq k_i$ and fix graphs $G_i\in\{P_{k_i}^{d_i},C_{k_i}^{d_i}\}$ for each $i\in[n]$.
    Define $X_i=I(d_i/k_i)$ if $G_i=P_{k_i}^{d_i}$ and $X_i=S(d_i/k_i)$ if $G_i=C_{k_i}^{d_i}$.
    Then
    \[
        \iso_{G_1\boxtimes\dots\boxtimes G_n}(x)\geq\biggl(\prod_{i=1}^n{k_i}\biggr)\cdot\iso_{X_1\boxtimes\dots\boxtimes X_n}\biggl({x\over\prod_{i=1}^n k_i}\biggr).
    \]
\end{theorem}
\begin{proof}
    We may label the vertices of $G_i$ as $\{0,\dots,k_i-1\}$.
    Consider the map $f\colon V(G_1\boxtimes\dots\boxtimes G_n)\to 2^{[0,1)^n}$ given by $f(x_1,\dots,x_n)=\prod_{i=1}^n\bigl[{x_i\over k_i},{x_i+1\over k_i}\bigr)$.
    Extend $f$ to $2^{V(G_1\boxtimes\dots\boxtimes G_n)}$ by defining $f(S)=\bigcup_{s\in S}f(s)$.
    Observe that $\lebesgue^{k}(f(S))=\abs S/\prod_{i=1}^n k_i$ for any $S\subseteq V(G_1\boxtimes\dots\boxtimes G_n)$.
    Furthermore, it is quick to observe that $f(N_{G_1\boxtimes\dots\boxtimes G_n}[S])=N_{X_1\boxtimes\dots\boxtimes X_n}[f(S)]$, up to a set of measure 0.
    Thus the claim follows.
\end{proof}
Naturally, equality does not need to hold in the above theorem unless certain divisibility conditions relating $x$ to $d_1,\dots,d_n,k_1,\dots,k_n$ hold.
However, equality does hold infinitely often.
\medskip

As a final result, we note that the celebrated isoperimetric inequality of Bollob\'as--Leader~\cite{bollobas_edge} is a corollary of our results.

The (restricted) $\ell_p$ Minkowski content of a measurable subset $A\subseteq[0,1]^n$ is defined to be
\[
    M_p(A)\eqdef\lim_{\epsilon\to 0^+}{\lebesgue^n\bigl((A+\epsilon B_p)\cap[0,1]^n\bigr)-\lebesgue^n(A)\over\epsilon},
\]
where $B_p$ is the $\ell_p$ unit ball.

\begin{theorem}\label[theorem]{minkowskiContent}
    For any measurable $A\subseteq[0,1]^n$, there is a cuboid or anticuboid $X\subseteq[0,1]^n$ with $\lebesgue^n(X)=\lebesgue^n(A)$ and $M_\infty(X)\leq M_\infty(A)$.
\end{theorem}
\begin{proof}
    Fix $0<\epsilon<1$ and set $\alpha=\lebesgue^n(A)$.
    Then $(A+\epsilon B_\infty)\cap[0,1]^n$ is precisely $N_{I(\epsilon,\dots,\epsilon)}[A]$ and so $\lebesgue^n\bigl((A+\epsilon B_\infty)\cap[0,1]^n\bigr)\geq\iso_{I(\epsilon,\dots,\epsilon)}(\alpha)$.
    Due to \Cref{isoperimetricBound}, we can find a sequence $\epsilon_m\to 0$ and some fixed $t\in\{0,\dots,d\}$ such that either $\iso_{I(\epsilon_m,\dots,\epsilon_m)}(\alpha)=\bcube{\underbrace{\epsilon_m,\dots,\epsilon_m}_t}(\alpha)$ for all $m$ or $\iso_{I(\epsilon_m,\dots,\epsilon_m)}(\alpha)=\banticube{\underbrace{\epsilon_m,\dots,\epsilon_m}_t}(\alpha)$ for all $m$.
    In either case, we can find either a cuboid or anti-cuboid $X$ with $\lebesgue^n(X)=\alpha$ and $\lebesgue^n(N_{I(\epsilon_m,\dots,\epsilon_m)}[X])=\iso_{I(\epsilon_m,\dots,\epsilon_m)}(\alpha)$ for all $n$.
    Therefore, $M_\infty(X)\leq M_\infty(A)$, which concludes the proof.
\end{proof}

A subset $A\subseteq[0,1]^n$ is said to be \emph{rectilinear} if it can be written as the finite union of boxes.
\begin{prop}\label[prop]{rectilinear}
    If $A\subseteq[0,1]^n$ is rectilinear, then $M_p(A)=M_\infty(A)$ for all $p\geq 1$.
\end{prop}
The proof is routine and we relegate it to \Cref{sec:rectilinear}. Note that this does does not hold for non-rectilinear sets.

With \Cref{rectilinear} in hand, and the observation that cuboids and anticuboids themselves are rectilinear, \Cref{minkowskiContent} immediately implies the celebrated isoperimetric inequality of Bollob\'as and Leader.
\begin{corollary}[Bollob\'as--Leader~\cite{bollobas_edge}]
    If $A\subseteq[0,1]^n$ is rectilinear, then there is a cuboid or anticuboid $X\subseteq[0,1]^n$ with $\lebesgue^n(X)=\lebesgue^n(A)$ and $M_p(X)\leq M_p(A)$ for every $p\geq 1$.
\end{corollary}

We close with a general question regarding variants of \Cref{mainCube}. What is the minimum possible $\lebesgue(N_{\delta}[A])$ among $A \subset [0,1]^n$, if $N_\delta[A] \eqdef \{x \in [0,1]^n: ||x-a||_p< \delta \text{ for some } a \in A\}$, for other $\ell_p$ norms? \Cref{mainCube} answers this precisely for when $p=\infty$. Another theorem of Bollob\'as--Leader \cite{bollobas1991vertex} resolves the graph isoperimetric problem for the (usual) grid graph on $[k]^n$, which by an argument analogous to the proof of \Cref{thm:discretize}
likewise answers this question in full for the case $p=1$. The $p=2$ case is elusive: indeed, it is still unknown (see e.g.\ \cite{bollobas_edge, chambers2023_3D, ritore1996spaces, ros2001isoperimetric}) what the smallest $M_2(X)$ is among all (not necessarily rectilinear) $X \subseteq [0,1]^n$ of given volume. An answer to our question would likely also answer this question by another limiting argument as in the proof of \Cref{minkowskiContent}.

\bibliographystyle{abbrv}
\bibliography{references}

@article{bollobas1991vertex,
  title={Compressions and isoperimetric inequalities},
  author={Bollob{\'a}s, B{\'e}la and Leader, Imre},
  journal={Journal of Combinatorial Theory, Series A},
  volume={56},
  number={1},
  pages={47--62},
  year={1991},
  publisher={Elsevier}
}

@article{bollobas_edge,
  doi = {10.1007/bf01275667},
  url = {https://doi.org/10.1007/bf01275667},
  year = {1991},
  month = dec,
  publisher = {Springer Science and Business Media {LLC}},
  volume = {11},
  number = {4},
  pages = {299--314},
  author = {B{\'e}la Bollob{\'a}s and Imre Leader},
  title = {Edge-isoperimetric inequalities in the grid},
  journal = {Combinatorica}
}

@article{chambers2023_3D,
  title={On the relative isoperimetric problem for the cube},
  author={Chambers, Gregory R and Mouill{\'e}, Lawrence},
  journal={arXiv preprint arXiv:2302.04382},
  year={2023}
}

@article{harper1999l0,
  title={On an isoperimetric problem for Hamming graphs},
  author={Harper, Lawrence H},
  journal={Discrete applied mathematics},
  volume={95},
  number={1-3},
  pages={285--309},
  year={1999},
  publisher={Elsevier}
}

@article{radcliffe_vertex,
  title={Vertex isoperimetric inequalities for a family of graphs on $\mathbb{Z}^k$},
  author={Radcliffe, J and Veomett, E},
  journal={Electronic Journal of Combinatorics},
  volume={19},
  number={2},
  pages={P45},
  year={2012}
}

@article{ritore1996spaces,
  title={The spaces of index one minimal surfaces and stable constant mean curvature surfaces embedded in flat three manifolds},
  author={Ritor{\'e}, Manuel and Ros, Antonio},
  journal={Transactions of the American Mathematical Society},
  volume={348},
  number={1},
  pages={391--410},
  year={1996}
}

@article{ros2001isoperimetric,
  title={The isoperimetric problem},
  author={Ros, Antonio},
  journal={Global theory of minimal surfaces},
  volume={2},
  pages={175--209},
  year={2001}
}

@article{wang2025_2D,
  title={Discrete isoperimetric inequalities on the strong products of paths},
  author={Wang, Runze},
  journal={arXiv preprint arXiv:2502.12199},
  year={2025}
}

@book{dudley_book,
  author    = {R. M. Dudley},
  title     = {Real Analysis and Probability},
  series    = {Cambridge Studies in Advanced Mathematics},
  volume    = {74},
  publisher = {Cambridge University Press},
  year      = {2002}
}

@book{kechris_book,
  author    = {Alexander S. Kechris},
  title     = {Classical Descriptive Set Theory},
  series    = {Graduate Texts in Mathematics},
  volume    = {156},
  publisher = {Springer},
  year      = {1995}
}

\appendix
\section{Proof of \Cref{minWithDifferentDomains}}\label[appendix]{sec:minWithDifferentDomains}

Define $g\colon I_n\to\R$ by $g\eqdef\min_{i\in[n]}\breve f_i$.
Since $I_1\subseteq I_2\subseteq\cdots\subseteq I_n$, we can equivalently write
\[
    g(x)=\min_{\substack{i\in[n]:\\ x\in I_i}}f_i(x).
\]
We begin with a quick observation:
\begin{equation}\label{eqn:interior}
    g(x)=\min_{\substack{i\in[n]:\\ x\in I_i^\circ}}f_i(x),\qquad\text{for all }x\in I_n^\circ.
\end{equation}
Here, $I^\circ$ denotes the interior of the interval $I$.
To see why this is the case, fix $x\in I_n^\circ$ and let $T\subseteq[n]$ be the set of all $i\in[n]$ for which $f_i(x)=g(x)$; we need to show that $x\in I_i^\circ$ for some $i\in T$.
Suppose this to be false; hence, $x\in\partial I_i$ for all $i\in T$.
Set $t=\max T$, so we know that $x\in\partial I_t\setminus\partial I_n$ (since $x\in I_n^\circ$); therefore, \eqref{cond:smaller} implies that either $f_t(x)>\min_{j\in[n]}\breve f_j(x)$
or
$f_t(x)\geq\min_{j\in\{t+1,\dots,n\}}f_j(x)$.
The latter is impossible
since $t=\max T$
and the former is impossible
by the definition of $T$; we have thus justified \eqref{eqn:interior}.
\medskip

We now establish the first item: if $x\in I_n$ is such that each $f_i$ for which $x\in I_i$ is continuous at $x$, then $g$ is also continuous at $x$.
It is well-known that the minimum of continuous functions on a common domain is also continuous, so the claim is immediate if $x\in\partial I_n$; thus suppose that $x\in I_n^\circ$.
Without loss of generality, we may suppose that if $x\in\partial I_i$, then $x$ is a right-endpoint of $I_i$.
Let $T\subseteq[n]$ be the set of all $i\in[n]$ for which $x\in I_i^\circ$.
Additionally, let $T'\subseteq[n]$ be the set of all $i\in[n]$ for which $x\in I_i$.
Again, since the minimum of continuous functions on a common domain is also continuous, we know that $\min_{i\in T}f_i$ is continuous at $x$ and that $\min_{i\in T'}f_i$ is left-continuous at $x$.
Thus, the only way for $g$ to fail to be continous at $x$ is if $\min_{i\in T'}f_i(x)\neq\min_{i\in T}f_i(x)$, which is impossible due to \eqref{eqn:interior}.
\medskip

We next establish the second item, which is more interesting: if $f_i$ is concave on $I_i$, then $g$ is concave on $I_n$.
We begin with a ``stitching'' lemma:
\begin{lemma}\label[lemma]{concaveSteps}
    Fix numbers $a_1<a_2<\cdots<a_m$ and suppose that $h\colon[a_1,a_m]\to\R$ is concave on $[a_i,a_{i+1}]$ for all $i\in[m-1]$.
    For each $i\in\{2,\dots,m-1\}$ suppose that $H_i\colon[a_{i-1},a_{i+1}]\to\R$ is a concave function satisfying $h\leq H_i$ and $h(a_i)=H_i(a_i)$.
    Then $h$ is concave on $[a_1,a_m]$.
\end{lemma}
\begin{proof}
    Recall that a function $f$ is concave if and only if ${f(y)-f(x)\over y-x}\geq{f(z)-f(y)\over z-y}$ for all $x<y<z$.
    Fix any $a_1\leq x<y<z\leq a_m$ and consider the sequence $x=b_0<b_1<\dots<b_k=z$ where $b_1,\dots,b_{k-1}$ are exactly the elements of $a_1,\dots,a_m$ between $x$ and $z$ along with the element $y$; suppose that $y=b_t$.
    Define $s_i={h(b_i)-h(b_{i-1})\over b_i-b_{i-1}}$ for $i\in[k]$.
    We begin by claiming that $s_1\geq s_2\geq\dots\geq s_k$.
    To see this, fix $i\in[k-1]$ and suppose first that $b_i\notin\{a_1,\dots,a_m\}$.
    In this case, $[b_{i-1},b_{i+1}]\subseteq[a_j,a_{j+1}]$ for some $j$ and so $s_i\geq s_{i+1}$ since $h$ is concave on $[a_j,a_{j+1}]$.
    Suppose next that $b_i=a_j$ for some $j$, so $[b_{i-1},b_i]\subseteq[a_{j-1},a_j]$ and $[b_i,b_{i+1}]\subseteq[a_j,a_{j+1}]$.
    Now, using the assumptions on $h$ and $H_j$, we have
    \[
        s_i={h(a_j)-h(b_{i-1})\over a_j-b_{i-1}}\geq{H_j(a_j)-H_j(b_{i-1})\over a_j-b_{i-1}}\geq{H_j(b_{i+1})-H_j(a_j)\over b_{i+1}-a_j}\geq{h(b_{i+1})-h(a_j)\over b_{i+1}-a_j}=s_{i+1}.
    \]
    With this in-hand, setting $\delta_i=b_i-b_{i-1}$, we use the fact that weighted averages are bounded above and below by the largest and smallest elements, respectively, to bound
    \[
        {h(y)-h(x)\over y-x}={\sum_{i=1}^t\delta_is_i\over \sum_{i=1}^t\delta_i}\geq s_t\geq s_{t+1}\geq{\sum_{i=t+1}^k\delta_is_i\over\sum_{i=t+1}^k\delta_i}={h(z)-h(y)\over z-y}.\qedhere
    \]
\end{proof}
We now apply the lemma to our situation.
Let $a_1<a_2<\dots<a_m$ be the distinct endpoints of $I_1,\dots,I_n$ and, for $i\in\{2,\dots,m-1\}$, define $T_i\eqdef\{j\in[n]:a_i\in I_j^\circ\}$.
Observe that $j\in T_i$ if and only if $I_j\supseteq[a_{i-1},a_{i+1}]$.
From here, define the functions $G_i\colon[a_{i-1},a_{i+1}]\to\R$ by $G_i=\min_{j\in T_i}f_j$.
By construction, each $G_i$ is concave since the minimum of concave functions on a common domain is also concave; additionally, $G_i\geq g$ on $[a_{i-1},a_{i+1}]$ and $G_i(a_i)=g(a_i)$ due to \eqref{eqn:interior}.
Thus, we may apply \Cref{concaveSteps} with $h\gets g$ and $H_i\gets G_i$ to conclude that $g$ is concave on $[a_1,a_m]=I_n$.

\section{Proof of \Cref{rectilinear}}\label[appendix]{sec:rectilinear}

We need to show that
\[
    \lim_{\epsilon\to 0^+}{\lebesgue^n\bigl((A+\epsilon B_\infty)\cap[0,1]^n\bigr)-\lebesgue^n\bigl((A+\epsilon B_p)\cap[0,1]^n\bigr)\over\epsilon}=0.
\]
To prove this, since $B_\infty\supseteq B_p$, it suffices to show that
\[
    \lim_{\epsilon\to 0^+}{\lebesgue^n(A+\epsilon B_\infty)-\lebesgue^n(A+\epsilon B_p)\over\epsilon}\leq 0.
\]

Consider first the case when $A=\prod_{i=1}^d[x_i,y_i]$.
Naturally, $A+\epsilon B_\infty=\prod_{i=1}^d[x_i-\epsilon,y_i+\epsilon]$.
In particular,
\begin{align*}
    \lebesgue^n(A+\epsilon B_\infty)=\prod_{i=1}^d(y_i-x_i+2\epsilon)=\prod_{i=1}^d(y_i-x_i)+2\epsilon\sum_{t=1}^d\prod_{i\in[d]\setminus\{t\}}(y_i-x_i)+P\epsilon^2,
\end{align*}
where $P$ is some polynomial in $\epsilon$ and the $x_i$'s and $y_i$'s.

Additionally,
\[
    A+\epsilon B_p\supseteq\prod_{i=1}^{t-1}[x_i,y_i]\times[x_i-\epsilon,y_i+\epsilon]\times\prod_{i=t+1}^d[x_i,y_i]
\]
for every $t\in[d]$.
Thus,
\[
    \lebesgue^n(A+\epsilon B_p)\geq \prod_{i=1}^d(y_i-x_i)+2\epsilon\sum_{t=1}^d\prod_{i\in[d]\setminus\{t\}}(y_i-x_i)
\]
We conclude that
\[
    \lim_{\epsilon\to 0^+}{\lebesgue^n(A+\epsilon B_\infty)-\lebesgue^n(A+\epsilon B_p)\over\epsilon}\leq\lim_{\epsilon\to 0^+}{P\epsilon^2\over\epsilon}=0.
\]

Next, consider the case when $A=\bigcup_{i=1}^n R_i$ where each $R_i$ is a box.
Since $(A+\epsilon B_\infty)\setminus (A+\epsilon B_p) = \bigcup_{i=1}^n (R_i + \epsilon B_\infty) \setminus (A+\epsilon B_p)\subseteq\bigcup_{i=1}^n(R_i+\epsilon B_\infty)\setminus(R_i+\epsilon B_p)$,
\[
    \lim_{\epsilon\to 0^+}{\lebesgue^n(A+\epsilon B_\infty)-\lebesgue^n(A+\epsilon B_p)\over\epsilon}\leq\sum_{i=1}^n\lim_{\epsilon\to 0^+}{\lebesgue^n(R_i+\epsilon B_\infty)-\lebesgue^n(R_i+\epsilon B_p)\over\epsilon}=0.\qedhere
\]

\end{document}